\documentclass[11pt, a4paper, article,reqno,]{amsart}
\usepackage{dsfont}
\usepackage{graphicx, xcolor,amsmath, amssymb}
\usepackage{amsmath,stackengine}
\usepackage{geometry}
\usepackage{appendix}
\usepackage{mathrsfs}
\usepackage{amssymb,latexsym}
\usepackage{amsthm}

\newtheorem{theorem}{Theorem}[section]
\newtheorem{proposition}[theorem]{Proposition}
\newtheorem{definition}[theorem]{Definition}
\newtheorem{lemma}[theorem]{Lemma}
\newtheorem{corollary}[theorem]{Corollary}

\theoremstyle{remark}
\newtheorem{remark}[theorem]{Remark}
\DeclareMathOperator{\dive}{div}

\newenvironment{acknowledgments}{%
  \bigskip
  \noindent\textbf{Acknowledgments.}\quad}{\par}

\numberwithin{equation}{section}

\title[\empty]{Global Existence and Uniqueness of Strong Solutions for a Phase Transition Model in Atmospheric Dynamics} 

\author[G. Cianfarani Carnevale]{Giada Cianfarani Carnevale}
\address[Giada Cianfarani Carnevale]{Dipartimento di Ingegneria e Scienze dell'Informazione e Matematica, Universit\`a degli Studi dell'Aquila (Italy)}
\email{giada.cianfaranicarnevale@univaq.it}
\author[D. Donatelli]{Donatella Donatelli}
\address[Donatella Donatelli]{Dipartimento di Ingegneria e Scienze dell'Informazione e Matematica, Universit\`a degli Studi dell'Aquila (Italy)}
\email{donatella.donatelli@univaq.it}
\author[S. Spirito]{Stefano Spirito}
\address[Stefano Spirito]{Dipartimento di Ingegneria e Scienze dell'Informazione e Matematica, Universit\`a degli Studi dell'Aquila (Italy)}
\email{stefano.spirito@univaq.it}


\allowdisplaybreaks

\begin{document}

\begin{abstract}
   In this work, we study a phase transition model in atmospheric dynamics, inspired by the works \cite{CZ, LT, LT2}, which analyze the primitive equations governing the evolution of velocity, temperature, and specific humidity. The main difficulty arises from the presence of a multivalued discontinuous nonlinear term in the temperature and in the humidity equations, describing the formation of precipitations, which becomes active under supersaturation conditions. To overcome this issue, we introduce a regularized formulation that ensures the existence and uniqueness of approximate solutions. By employing classical compactness arguments, we then establish the existence of a strong solution to the original model. Additionally, we establish uniqueness under a conditional and physically meaningful assumption. This approach allows us to provide a rigorous justification of the tropical climate model on the whole space $\mathbb{R}^2$, while avoiding the introduction of a viscosity term in the humidity equation.
\end{abstract}

\maketitle

\section{Introduction}\label{sec:intro}

The fundamental equations of atmospheric dynamics, known as the primitive equations, are derived from the Navier-Stokes equations with rotation. These equations are integrated with thermodynamic diffusion, considering that the ratio of the vertical scale to the horizontal scale is negligible. The primitive equations account for buoyancy forces and stratification effects, assuming the Boussinesq hypothesis, which implies the neglect of density variations in all terms except those multiplied by g, the acceleration due to gravity. On large scales, atmospheric flow can be uniquely decomposed into its barotropic and baroclinic components. These approaches are based on the characteristic scales of motion: horizontal scales, which extend over thousands of kilometers, are much larger than the vertical scales of the troposphere, approximately 16 km (\cite{KM}). Consequently, vertical motion is modeled through hydrostatic balance, expressed as 
\[
\frac{\partial p}{\partial z} = -\rho g,
\]
where \( p \) represents pressure and \( \rho \) the density. Following the derivation in \cite{KM} and considering the Boussinesq and hydrostatic assumptions, the viscous, incompressible primitive equations derived from the Navier Stokes equations can be written as follows:

\begin{equation}\label{prim_intro}
	\left\{\begin{aligned}
		& \partial_t V + V \cdot \nabla V + W\partial_z V - \mu \Delta_h V - \nu \partial_z^2 V + \nabla_h \Phi =0\\
  & \dive_h V + \partial_z W = 0 \\
		&  \partial_z \Phi = \frac{g}{\theta_0}\theta \\
	&\partial_t \theta + V \cdot \nabla_h \theta + W \partial_z \theta + \frac{N^2\theta_0}{g}W = S_{\theta} \\
 &\partial_t Q + V \cdot \nabla_h Q + W \partial_z Q = - P 
		\end{aligned}\right.
\end{equation}
where $V=(V_1,V_2)$, $U=(V,W)$ is the velocity field, $\Phi$ the pressure, $Q$ the moisture, $h=(x,y)$, $\theta$ the potential temperature (the temperature that a parcel of air would have if displaced adiabatically in the vertical from its original position to a fixed reference position). Classic references on this subject include \cite{LTW,LTW2,P,R,salmon,TZ}.
 This small aspect ratio limit has been rigorously justified, see for example  \cite{AG01,  LT19} for the standard atmosphere models. In  \cite{DJ, DJ2} the small aspect ratio limit has been rigorously performed to model  the polluted atmosphere as a geophysical fluid sharing with seas and oceans a shallow horizontal domain where,  to the fluid equations,  a convection and diffusion equation is added, aimed at modelling the presence of pollutants in the air and  the ``downwind-matching'' coordinate system,  better suited to the meteorological context, is adopted.

The global existence of the primitive viscous equations, fundamental for the large-scale description of oceanic and atmospheric dynamics, has been shown in the seminal paper \cite{CT}. In \cite{LT_chap} there is an overview regarding the most relevant results in the existence theory for the primitive equations and the recent developments in geophysical models:  global stability of strong solutions (with full or partial viscosity),  conditional uniqueness of weak solutions and their global stability, even in the presence of discontinuous initial data and  singular perturbation limits such as small aspect ratio which rigorously justify hydrostatic equilibrium  and relaxation limits. 

During the last years a lot of effort has been devoted to the development of the atmosphere modelling addressing different issues such as, for example, the role of full diffusivity and viscosity  (\cite{CLT14}, \cite{CLT16}).  Some other relevant aspects in the atmosphere modelling are phase transition phenomena due to air saturation and condensation, therefore when moisture is included, an equation for the conservation of water must be added,  \cite{LTW, GH06}. In the past, the equation of conservation of moisture was simply an equation of transport that did not account for changes in phase, concentration, evaporation, and rainfall. Recently, researchers have addressed models that incorporate these changes in phase to provide a rigorous mathematical framework for studying such systems of equations. In \cite{HKLT} it is analysed the global solvability of the primitive equations of the atmosphere, integrated with a humidity model that includes phase changes (condensation, evapotranspiration, rainfall formation) and processes of auto-transformation between clouds and rain. In  \cite{BCZ,CZ,LT}  the authors propose primitive equations which describe the evolution of velocity, temperature, and specific humidity (the mass of water vapour per unit of air). The main difficulty lies in the presence of discontinuous multivalued nonlinear term in the temperature and humidity equations, which is active only under supersaturation conditions, and in the corresponding variations of temperature. To address the problem, both in \cite{BCZ,CZ}, the authors reformulate the system using the pressure coordinates $(x,y,p)$ thanks to the hydrostatic assumption, and make use of differential inclusions to show the existence of weak solutions. Contrary to \cite{BCZ}, in \cite{CZ} the velocity is not prescribed and the global existence of quasi-strong and strong solutions is proved, along with results on uniqueness and maximum principles that are physically relevant. 

Strictly related to the moisture phase transition model is the so called \textit{Tropical Climate Model} introduced in  \cite{FM}. 
The absence of the Coriolis force at the equator favours the formation of different types of atmospheric waves, characterised by complex vertical and meridional structures. These waves are mainly divided into two categories: barotropic waves, which propagate toward the poles, and baroclinic waves, which remain localised near the equator and extend along it.
Moreover, in  the tropical zone of the earth, the wind in the lower troposphere is of equal magnitude, but with opposite sign to that in the upper troposphere, in other words, the primary effect is captured in the first baroclinic mode, i.e.on the first mode of the fluctuation of the solution about its vertical average. However, for the study of the tropical-extratropical interactions, where the transport of momentum between the barotropic (the spatial vertical average of the solution) and baroclinic (the fluctuation of the solution about the barotropic part) modes plays an important role, it is necessary to retain both the barotropic and baroclinic modes of the velocity.  In \cite{LT2}, the  global well-posedness of strong solutions was established for this kind of model and in  \cite{LSL} the existence and uniqueness of global solutions for the three-dimensional tropical climate model with partial viscosity and damping are established. Techniques involving anisotropic embedding and fractional derivative interpolation inequalities are employed to obtain estimates and smooth solutions, which are compared with the classical Navier Stokes equations.  An important ingredient of the tropical atmospheric circulation is the water vapour. Water vapour is the most abundant greenhouse gas in the atmosphere, and it is responsible for amplifying the long-term warming or cooling cycles. Therefore, one should also consider the coupling with an equation modeling moisture in the atmosphere. Indeed in \cite{FM} the equations are analogous to those used for shallow waters and include a precipitation term that acts as a nonlinear source. This system is also coupled with the nonlinear, incompressible barotropic flow equations. From the perspective of tropical atmospheric dynamics, this theory offers a new outlook on how large regions of humidity, associated with deep convection, can move and interact with large-scale dynamics within a near-equilibrium framework. In \cite{LT}, based on the work of \cite{FM}, a nonlinear system describing the interaction between the barotropic mode and the first baroclinic mode of the tropical atmosphere with humidity is analyzed. The existence and uniqueness of strong solutions for initial data in $H^1$ are shown for every positive convective relaxation parameter $\epsilon$. Moreover, if the initial data have higher regularity than $H^1$, the solution depends continuously on the data. It is also shown that, as $\epsilon$ tends to zero, the system converges to a limit system at a rate of $O(\sqrt{\epsilon})$. This limit system, in turn, possesses a unique strong solution in $H^1$ that depends continuously on the data if they are slightly more regular than $H^1$. 

The purpose of this paper is to fill the gap in the models \cite{LT} and \cite{CZ}. We analyze a tropical climate model that incorporates the barotropic and baroclinic modes, as in \cite{LT}, but introduces the multivalued discontinuous nonlinear term present in \cite{CZ} into the temperature and humidity equations. This term better describes the ratio of vapor in the air, as it is crucial to account for the potential saturation of vapor, which leads to condensation (clouds) and rain. Unlike \cite{CZ}, no diffusion is required in the humidity equation. However, from a physical perspective, the temperature diffusion mechanism cannot be ignored, and we consider diffusivity in the temperature equation, contrary to \cite{LT}.

 The paper is organized as follows. In Section \ref{sec:2} some technical lemmas and basic notations are reviewed. In Section \ref{sec3}, following the work of \cite{FM} and starting from \eqref{prim_intro} we derive the so called \textit{Tropical Climate Model} (TCM): 
\begin{equation}\label{tro_model_fin_intro}
     \left\{\begin{aligned}
&\partial_t u + u \cdot \nabla u - \Delta u + \nabla p + \dive (v \otimes v) = 0\\
& \partial_ t v + u \cdot \nabla v + v \nabla u - \Delta v = \frac{H}{\pi} \frac{g}{\theta_0} \nabla T \\
& \dive u = 0 \\
& \partial_t T + u \nabla T - \Delta T - \frac{N^2 \theta_0}{g} \dive v =  \frac{H g}{\pi R}(\dive v)^- \mathcal{H}(q-q_s) G^+(T)\\
& \partial_t q + u \nabla q + \bar{Q} \dive v = -  \frac{H g}{\pi R}(\dive v)^- \mathcal{H}(q-q_s) G^+(T),
    \end{aligned}\right.
\end{equation}
where $u$ is the \textit{barotropic velocity}, $v$ the \textit{baroclinic velocity}, $T$ the \textit{temperature}, $q_s \in (0,1)$ the \textit{saturation humidity} and $\mathcal{H}(q-q_s)$ the multivalued Heaviside function.

In Section \ref{sec:mainresults} we state precisely the main results of the paper. In Section \ref{sec:4}, we introduce a family of regularized problems that approximate system \eqref{tro_model_fin_intro}. First, we regularize the Heaviside function via a mollification $\mathcal{H}_\epsilon(q-q_s)$; then, we add an artificial viscosity term $\eta \Delta q$ in the moisture equation. We establish the global well-posedness for each fixed pair of positive parameters $\epsilon$ and $\eta$.

In Section \ref{sec:5}, we collect all $\eta$-uniform estimates and perform the limit as $\eta \to 0$, thereby proving the global well-posedness for every fixed $\epsilon>0$.

Finally, in Section \ref{sec:6}, we take the limit as $\epsilon \to 0$ and prove the existence and uniqueness of a strong solution to the original system \eqref{tro_model_fin_intro}. Moreover, we establish that the temperature $T$ remains essentially bounded in space and time. This ensures the exclusion of unbounded temperature growth, consistently with physical expectations in tropical environments.

\section{Notations and Technical Lemmas}\label{sec:2}
We denote by \( L^p(\mathbb{R}^2) \) the classical Lebesgue space for \( p \in [1, \infty] \), equipped with the norm \( \|\cdot\|_p \). The space \( W^{l,p}(\mathbb{R}^2) \) consists of functions that are integrable in the \( L^p \) sense, and whose weak derivatives up to order \( l \) are also in \( L^p(\mathbb{R}^2) \). The norm of an element \( f \in W^{l,p}(\mathbb{R}^2) \) is defined as $\|f\|_{W^{l,p}(\mathbb{R}^2)} = \sum_{|\alpha| \leq l} \| D^{\alpha}f \|_{L^{p}(\mathbb{R}^2)}$. We denote as \( W^{k,p}(0,T;X) \) the function spaces consisting of functions \( f: [0,T] \to X \) (where \( X \) is a Banach space) that are \( k \) times weakly differentiable with derivatives in \( L^p(0,T;X) \). 

Given a positive time $T>0$, for a sufficiently regular vector field \( f : \mathbb{R}^2 \times (0, T) \to \mathbb{R}^2 \), we denote by \( \operatorname{div} f \) its divergence and by \( \Delta f \) its Laplacian as
\begin{equation*}
  \operatorname{div} f = \partial_x f_1 + \partial_y f_2, \quad \text{and} \quad \Delta f = (\partial_{xx} f_1 + \partial_{yy} f_1, \; \partial_{xx} f_2 + \partial_{yy} f_2), 
\end{equation*}
where \( f = (f_1, f_2) \).

Finally, we will denote by $\langle \cdot \ , \ \cdot\rangle$ the following averaged integral on the vertical variable

$$\langle f, g\rangle = \frac{1}{H}\int_{0}^{H}fgdz.$$

Moreover  we recall some technical results that will be used in the next sections. The following lemma will be used to prove the uniqueness of strong solutions for system \eqref{tro_model_fin_intro}.

\begin{lemma}(\cite{LT})\label{lemma_uniqueness}
    Given a positive time $\mathcal{T}$, and let $m_1,m_2$ and $S>0$ be non-negative functions on $(0,\mathcal{T})$ such that
    \begin{equation*}
        m_1, S \in L^1((0,\mathcal{T})), \qquad m_2 \in L^2((0,\mathcal{T})), \qquad S>0 \; a.e. \; on \; (0,\mathcal{T}).
    \end{equation*}
    Suppose that $f$ and $G$ are two non-negative functions on $(0,\mathcal{T})$, with $f$ being absolutely continuous on $[0,\mathcal{T})$, that satisfy
    \begin{equation*}
         \left\{ \begin{aligned}
             & f'(t) + G(t) \leq m_1(t)f(t) + m_2\left[f(t) G(t) \log^+ \left( \frac{S(t)}{G(t)}\right) \right]^{\frac{1}{2}} \; a.e. \; on \; (0,\mathcal{T}) \\
             &f(0)=0
         \end{aligned} \right.
    \end{equation*}
    where $\log^+z = \max\left\{0, \log z\right\}$ for $z \in (0,\infty)$, and when $G(t)=0$, at some time $t \in [0,\mathcal{T})$, we adopt the following natural convection
    \begin{equation}
        G(t)\log^+ \left( \frac{S(t)}{G(t)} \right) = \lim_{z \rightarrow0^+}\log^+ \left( \frac{S(t)}{z}\right).
    \end{equation}
    Then, we have $f \equiv 0$ on $[0,\mathcal{T})$.
\end{lemma}

\begin{lemma}\label{egoroff}
    Let $\Omega \subseteq \mathbb{R}^d$ be a measurable set of positive measure and $f$ be a measurable function defined on $\Omega$. Suppose that for any positive number $\eta$ there is a measurable subset $E_\eta$ of $\Omega$ with $|E_\eta| \leq \eta$, such that $f=0$ a.e on $\Omega \setminus E_\eta$. Then, $f=0$ a.e on $\Omega$.
\end{lemma}

\begin{lemma}\label{A3}
Given a time $\mathcal{T} \in (0,\infty)$ and a number $p \in (1, \infty)$, let us consider $f \in L^2((0,\mathcal{T});L^p(\mathbb{R}^2))$ and $v$ be the unique solution to
\begin{equation*}
    \left\{ \begin{aligned}
        & \partial_t v - \Delta v = f \\
        & v|_{t=0} = v_0 \in H^1(\mathbb{R}^2).
    \end{aligned} \right.
\end{equation*}
Then we have the following estimate
\begin{equation*}
    \int_0^\mathcal{T} \|\Delta v\|_p dt \leq C_p( 1 + \sqrt{\mathcal{T}}) \left[ \|\nabla v_0\|_2 + \left( \int_0^\mathcal{T} \|f\|_p^2 dt\right)^{\frac{1}{2}} \right]
\end{equation*}
where $C_p$ is a positive constant depending only on $p$, and in particular is independent of $\mathcal{T}$, $f$ and $v_0$.
\end{lemma}

\section{Tropical climate model with saturation and humidity}\label{sec3}

The incompressible viscous Primitive Equations in the layer $\mathbb{R}^2 \times (0,H)$, where $H>0$ denote the thickness of the troposphere are the following,
\begin{equation}\label{pe-lt}
\left\{\begin{aligned}
		& \partial_t V + V \cdot \nabla V + W\partial_z V - \mu \Delta V - \nu \partial_z^2 V + \nabla \Phi =0\\
  & \dive V + \partial_z W = 0 \\
		&  \partial_z \Phi = \frac{g}{\theta_0}\theta \\
	&\partial_t \theta + V \cdot \nabla \theta + W \partial_z \theta + \frac{N^2\theta_0}{g}W = S_{\theta} \\
 &\partial_t Q + V \cdot \nabla Q + W \partial_z Q = - P 
		\end{aligned}\right.
\end{equation}
where $V=(V_1,V_2)$, $U=(V,W)$ is the velocity field, $\Phi$ the pressure, $Q$ the moisture, $P=P(x,y,t)$ the precipitation term, $S_{\theta}$ the temperature source which combines the heat flux with the precipitation. We recall that the total potential temperature $ \theta^{total}$ (full temperature including the background state) is given by
\begin{equation*}
    \theta^{total} = \theta_0 + \bar{\theta}(z) + \theta (x,y,z,t)
\end{equation*}
where $\theta_0>0$ is a mean basic state constant temperature,  $ \theta$ is the perturbation temperature, $\bar{\theta}(z)$ is the vertical profile background stratification which satisfies $N^{2}=\bar{\theta}(z)>0$ where $N$ is the Brunt-Vaisala buoyancy frequency which we assume to be constant. 
\subsection{Galerkin expansion}
In order to study the tropical-extra tropical interactions, following the work of \cite{FM} and \cite{LT} we perform the Galerkin projection of the primitive equations in the vertical direction onto the barotropic mode and the first baroclinic mode. 

We begin by performing a vertical decomposition of the variables in barotropic and baroclinic parts as:
\begin{equation}\label{decomp}
\left\{\begin{aligned}
    & V = u+ u' \\
    & \Phi = p + p' \\
    & Q= q+ q' \\
    & W = w' \\
    & \theta = T',
    \end{aligned}\right.
\end{equation}
where the depth-independent barotropic modes  are given by
\begin{equation*}
    u = < V, 1> ,\qquad p = < \phi, 1>,\qquad q = < Q, 1>
\end{equation*}
and  the baroclinc modes $(u',q',p')$ satisfy
\begin{equation*}
    < u', 1> =< q',1> = < p',1>= 0.
\end{equation*} 
The boundary conditions assumed here are no normal flow at the top and bottom of the troposphere, 
\begin{equation}
w'=0 \qquad \text{ at } z=0, H.
\end{equation}
From the hydrostatic relation $\eqref{pe-lt}_3$ we easily get:
\begin{equation}\label{newhydro}
    \frac{\partial p'}{\partial z} = \frac{g}{\theta_0} T'.
\end{equation}
The projected free tropospheric nonlinear dynamics obtained by substituing \eqref{decomp} in \eqref{pe-lt} are given by

\begin{equation}\label{baro}
\left\{\begin{aligned}
    &\frac{\bar{D}u}{Dt} + < u' \nabla u', 1> + <w'\frac{\partial v}{\partial z},1> = -\nabla p + \mu \Delta u  \\
    & \partial_t q + u \nabla q + < u' \nabla q', 1 > + < w' \frac{\partial q'}{\partial z}, 1> = -<P,1>,
    \end{aligned}\right.
\end{equation}
where we define
\begin{equation*}
    \frac{\bar{D}}{Dt} = \partial_t + \bar u \cdot \nabla.
\end{equation*}
The baroclinic dynamics is computed by taking the difference between \eqref{pe-lt} written in terms of the decomposition \eqref{decomp} and the projected barotropic part \eqref{baro}. We end up with
\begin{equation}\label{barocli}
    \left\{\begin{aligned}
    &\frac{\bar{D}u'}{Dt} + u' \nabla u + u' \nabla u' - < u' \nabla u', 1> + w' \frac{\partial u'}{\partial z} - <w' \frac{\partial u'}{\partial z},1> = -\nabla p' + \Delta u'  \\
   & \frac{\bar{D}T'}{Dt} + u' \nabla T' + w' \partial_z T' + \frac{N^2\theta_0}{g} w' = S_{T'} \\
   &  \partial_t q' + u \nabla q' + u' \nabla q + u' \nabla q' - < u'\nabla q', 1> + w' \partial_z q' - < w' \partial_z q'> = 0
    \end{aligned}\right.
\end{equation}
Finally we perform a  truncated Galerkin expansion of the form, 
\begin{equation}\label{gal}
\begin{aligned}
		& \begin{pmatrix}
  V \\
  \Phi \\
  Q \end{pmatrix} = \begin{pmatrix}
  u \\
  p  \\
  q \end{pmatrix} (x,y,t) + \begin{pmatrix}
  v (x,y,t) \\
  p_1 (x,y,t) \\
  \bar{Q}\end{pmatrix}  \sqrt{2}\cos(\pi z /H)\\
  & \begin{pmatrix}
  W \\
  \theta \end{pmatrix} = \begin{pmatrix}
  w \\
  T \end{pmatrix} (x,y,t) \sqrt{2}\sin(\pi z /H) 
		\end{aligned}
\end{equation}
where $\bar{Q}>0$ is the prescribed gross moisture stratification (for example see \cite{LT} for detailed description) and the baroclinic terms are defined as
\begin{equation*}
     \begin{aligned}
        & u'= v\sqrt{2}\cos \left( \frac{\pi z}{H} \right), \qquad & w'= w\sqrt{2}\sin \left( \frac{\pi z}{H} \right), \\
        & p'= p_1\sqrt{2}\cos \left( \frac{\pi z}{H} \right),  \qquad &T'=T\sqrt{2}\sin \left( \frac{\pi z}{H} \right),
    \end{aligned}
\end{equation*}
$p'$ denotes the mean zero component of the pressure. Moreover, we impose that 
$$S_{T'} = S_T(x,y,t) \sqrt{2}\sin (\pi z/H),$$
see \cite{SM}. 
In view of the expansion \eqref{gal}, we project in the vertical direction equations \eqref{baro} and \eqref{barocli}.

\subsection*{Continuity equation}
Let us start with the continuity equation, by using \eqref{gal} we get:
\begin{equation}\label{cont}
    \left\{\begin{aligned}
& \dive u = 0 \\
& w = -\frac{H}{\pi} \dive v 
    \end{aligned}\right.
\end{equation}

\subsection*{Hydrostatic balance}
Using $\eqref{gal}$ in $\eqref{newhydro}$ the hydrostatic balance becomes:
\begin{equation}\label{hydro}
    p_1 = - \frac{H}{\pi} \frac{g}{\theta_0} T.
\end{equation}

\subsection*{Momentum equation of barotropic velocity}
In view of \eqref{gal} we compute the projected terms in $\eqref{baro}_1$ and we get
\begin{equation}\label{barovel}
   \partial_t u + u \cdot \nabla u - \mu \Delta u + \nabla p + \dive (v \otimes v) = 0
   \end{equation}
\subsection*{Moisture equation}
Using $\eqref{gal}$ in $\eqref{baro}_2$ the moisture equation becomes:
\begin{equation}\label{gen_most}
   \partial_t q + u\nabla q + \bar{Q}\dive v = -P
\end{equation}
since $<P,1> = P$, $\nabla q'=0$ and 
\begin{equation*}
    < w' \frac{\partial q'}{\partial z}, 1 > = \frac{2}{H} \int_0^H \bar{Q} \dive v \sin^2\left(\frac{\pi z}{H} \right) dz = \bar{Q} \dive v.
\end{equation*}
Following the derivation of \cite{FM} we ignore the third equation in \eqref{barocli}, namely the equation of the baroclinic part of the moisture. 

\subsection*{Baroclinic velocity and first law of thermodynamics}
To derive the reduced dynamics for the baroclinic velocity $u'$ and the temperature $T'$ we take the
inner product of equations $\eqref{barocli}_1$, $\eqref{barocli}_2$  with $\sqrt{2} \cos \left( \frac{\pi z}{H} \right)$ and $\sqrt{2} \sin \left( \frac{\pi z}{H} \right)$ respectively, we get
\begin{equation}\label{baroclivel}
    \frac{\bar{D}v}{Dt} + v \cdot \nabla u = \frac{H}{\pi} \frac{g}{\theta_0} \nabla T + \mu \Delta v,
\end{equation}
and
\begin{equation*}
    \frac{\bar{D}T}{Dt} - \frac{H}{\pi}\frac{\theta_0 N^2}{g}\dive v = S_{T},
\end{equation*}
since $< S_{T'}, \sqrt{2} \sin \left( \frac{\pi z}{H} \right)> = S_T$. In order to close system one still needs to parameterize the source term  in the temperature equation $S_T$ and the precipitation $P$ in the moisture equation.

\subsection{Precipitation parameterization}\label{precipit}
The saturated adiabatic lapse rate is fundamental in understanding the structure and evolution of clouds, as well as the associated weather conditions. It influences atmospheric stability, precipitation, and other atmospheric phenomena. Air is considered saturated when the vapor pressure of humidity reaches its maximum limit, to the point where it cannot hold more water vapor without resulting in condensation, as in the case of cloud or fog formation. In this context, we will focus on adiabatic processes, meaning situations in which there is no heat exchange with the external environment. When a mass of air rises, it expands and cools; conversely, when it descends, it undergoes compression and warms up. We know from \cite{CZ} (and references therein) that the general form of moisture equation is the following:
\begin{equation}
    \frac{Dq}{Dt} = \frac{Dq_s}{Dt} + D_q=P
    \label{eq. for P}
\end{equation}
where \(q\) denotes the specific humidity, that is the ratio between the mass of water vapor and the total mass of moist air, and thus it is always assumed that \(q \geq 0\). The term $q_s$ denotes the saturation concentration (or saturation humidity) and $D_q$ is a suitable form of dissipation accounting for horizontal and vertical diffusion. 
For simplicity we impose from now on $D_q=0$, therefore our aim is to explicit through thermodynamic relations the form of $Dq_s/Dt$, hence of the precipitation term $P$. 

We denote by $L$ the latent heat of evaporization, $R$ is the ideal gas constant, $c_p$ is the air specific heat capacity, $R^{\nu}$ the gas constants for water vapour, $k=R^d/R^\nu$, $p$ the total pressure and $e_s$ the partial pressure of water vapour in air that occurs when the vapour is in equilibrium with the liquid. If more vapour were added, it would immediately condense. Starting from Chapter 18 in \cite{Vallis}, in the context of atmospheric convection processes, we review the most important thermodynamic relations to explicitly express the equations satisfied by the humidity $q$ and the temperature $T$. Here, we recall that in a moist atmosphere the moist static energy is conserved as a parcel ascends and this allows us to write down the equation satisfied by the temperature $T$ as
\begin{equation*}
    c_p \frac{dT}{dz}= -L \frac{dq}{dz} - g.
\end{equation*}
If the air is moist but not saturated then an ascending parcel will follow the dry adiabatic lapse rate (because $dq/dz= 0$) but if it is saturated, namely $q = q_s$, then as a parcel ascends it will cool and some water vapour will condense. Since $q_s$ can be expressed as $q_s= k e_s/p$, and so is a function of temperature and pressure, we have
\begin{equation*}
    \frac{dq_s}{dz}= \left( \frac{dq_s}{dT}\right)_p\frac{dT}{dz} + \frac{q_s}{p} \rho g,
\end{equation*}
using hydrostasy. Using the Clausius–Clapeyron equation 
\begin{equation*}
    \frac{de_s}{dT} = \frac{Le_s}{R^{\nu} T^2},
\end{equation*}
and the ideal gas equation of state we get the relation between $T$ and $q$ of an adiabatically ascending saturated parcel, that is
\begin{equation*}
    \left\{\begin{aligned}
& \frac{dT}{dz} = - \frac{L}{c_p} \frac{dq_s}{dz}- \frac{g}{c_p}\\
&  \frac{dq_s}{dz} = \frac{L q_s}{R^{\nu} T^2} \frac{dT}{dz} + \frac{g q_s}{R T}.
    \end{aligned}\right.
\end{equation*}
By substituting $\displaystyle{\frac{dT}{dz}}$ in the equation of $\displaystyle{\frac{dq_s}{dz}}$ we get:
\begin{equation*}
    \frac{dq_s}{dz} = \frac{g}{R} q_s \left( \frac{c_p R^{\nu} T- LR}{c_p R^{\nu} T^2 + L^2 q_s} \right),
\end{equation*}
and thus, by classical chain rule, we easily obtain:
\begin{equation*}
    \frac{Dq_s}{Dt} = \frac{dq_s}{dz} \frac{Dz}{Dt} = w \frac{g}{R} q_s \left( \frac{c_p R^{\nu} T- LR}{c_p R^{\nu} T^2 + L^2 q_s} \right).
\end{equation*}
Hence from \eqref{eq. for P} we get
$$P=- \frac{H}{\pi} \dive v  \frac{g}{R} q_s \left( \frac{c_p R^{\nu} T-LR}{c_p R^{\nu} T^2 + L^2 q_s} \right), $$
where we used \eqref{cont}.
Finally, the equation for $q$ is given by, 
\begin{equation}\label{eqq}
    \partial_t q + u \nabla q + \bar{Q} \dive v = \frac{H}{\pi} \dive v  \frac{g}{R} q_s \left( \frac{c_p R^{\nu} T-LR}{c_p R^{\nu} T^2 + L^2 q_s} \right). 
\end{equation}
Now, by recalling  that the contribution of the source term $S_T$ is given by a diffusivity part (heat flux) and a precipitation term, namely, $S_T = \Delta T +  P$ (see also \cite{LT,SM}) we get
\begin{equation}\label{eqT}
    \partial_t T + u \nabla T - \Delta T -  \frac{H}{\pi}\frac{\theta_0 N^2}{g}\dive v = -\frac{H}{\pi} \dive v  \frac{g}{R} q_s \left( \frac{c_p R^{\nu} T-LR}{c_p R^{\nu} T^2 + L^2 q_s} \right).
\end{equation}
Finally, collecting \eqref{cont}, \eqref{barovel}, \eqref{baroclivel}, \eqref{eqq} and \eqref{eqT} the system for tropical atmosphere  which describes humidity and moisture  interactions becomes:
\begin{equation}\label{tro_model}
     \left\{\begin{aligned}
&\partial_t u + u \cdot \nabla u - \mu \Delta u + \nabla p + \dive (v \otimes v) = 0\\
& \partial_ t v + u \cdot \nabla v + v \nabla u - \mu \Delta v = \frac{H}{\pi} \frac{g}{\theta_0} \nabla T \\
& \dive u = 0 \\
& \partial_t T + u \nabla T - \Delta T -  \frac{H}{\pi}\frac{\theta_0 N^2}{g} \dive v = - \frac{H}{\pi} \dive v  \frac{g}{R} q_s \left( \frac{LR- c_p R^{\nu} T}{c_p R^{\nu} T^2 + L^2 q_s} \right) \\
& \partial_t q + u \nabla q + \bar{Q} \dive v = \frac{H}{\pi} \dive v  \frac{g}{R} q_s \left( \frac{LR- c_p R^{\nu} T}{c_p R^{\nu} T^2 + L^2 q_s} \right)
    \end{aligned}\right.
\end{equation}

\begin{remark}
We recall that \(\bar{Q}\) represents the gross moisture stratification and thus one has to assume that \(\bar{Q} > 0\). This condition is motivated by the fact that typical humidity measurements indicate an exponential decrease of moisture with altitude over scales of a few kilometers. The parameter \(\bar{Q}\) plays a fundamental role in atmospheric dynamics: a higher value of \(\bar{Q}\) tends to reduce stability in humid regions, influencing convection processes.
\end{remark}

\subsection{Equivalent formulation}\label{formal}
In order to simplify the formulation of \eqref{tro_model} we impose that $\mu=1$ and the humidity saturation parameter \( q_s \) is a constant and  $0<q_{s}<1$.

In our setting the precipitation production $P$  is given by
\begin{equation}\label{sourceq}
    P= - \frac{H}{\pi} \dive v  \frac{g}{R} q_s \left( \frac{LR- c_p R^{\nu} T}{c_p R^{\nu} T^2 + L^2 q_s} \right) 
    = -\frac{H}{\pi} \dive v  \frac{g}{RT} F(T) = -\frac{H}{\pi} \dive v  \frac{g}{R}G(T)
\end{equation}
where $F(T)$ is defined as
\begin{equation*}
    F(T)= q_s T \left( \frac{LR - c_p R^{\nu}T}{c_p R^{\nu}T^2 + q_sL^2} \right),
\end{equation*}
and satisfies the following properties,
\begin{align}\label{prop_F}
    & |F(\xi_1)-F(\xi_2)|\leq M_F |\xi_1 -\xi_2| \quad \text{ for every }\xi_1,\xi_2 \in \mathbb{R}. \\
    &|F(\xi)| \leq C_F \quad \text{ for every }  \xi \in \mathbb{R}  \nonumber \\
    & \text{ Since } F(0)=0 \text{ we have } |F(\xi)| \leq M_F|\xi| \quad \text{ for every } \xi \in \mathbb{R}  \nonumber \\
   & \text{ Since } F(\xi_0) = 0 \text{ where } \xi_0 = LR/c_pR_\nu \text{ then } F(\xi) \geq 0 \Longleftrightarrow \xi \in [0, \xi_0]  \nonumber.
\end{align}
Moreover the function $G(T)$ is defined as 
\begin{equation*}
    G(T):= \frac{F(T)}{T} = q_s \left( \frac{LR- c_p R^{\nu} T}{c_p R^{\nu} T^2 + L^2 q_s} \right),
\end{equation*}
and the following properties hold
    \begin{align}\label{prop_G}
    & |G(\xi_1)-G(\xi_2)|\leq M_G |\xi_1 -\xi_2| \quad \text{ for every }\xi_1,\xi_2 \in \mathbb{R}. \\
    &|G(\xi)| \leq C_G \quad \text{ for every }  \xi \in \mathbb{R}  \nonumber \\
    & \text{ Since } G(0)=\frac{R}{L} \text{ we have } \left |G(\xi) - \frac{R}{L}\right | \leq M_G|\xi| \quad \text{ for every } \xi \in \mathbb{R}  \nonumber \\
   & \text{ Since } G(\xi_0) = 0 \text{ where } \xi_0 = LR/c_pR_\nu \text{ then } G(\xi) \geq 0 \Longleftrightarrow \xi \in (- \infty, \xi_0]  \nonumber.
\end{align}
When the air is humid, the convection process, which is the vertical movement of air, occurs more easily compared to dry air. This is because when the water vapor present in the air condenses, changing from a gaseous to a liquid state, it releases heat called ``latent heat''. This heat helps to warm the surrounding air, thus facilitating its uplift. However, this condition only occurs when the air is saturated. In other words, only when the air is fully moist and saturated the heat released during condensation can influence the stability of the air and promote the convection process. Therefore one assumes that the contribution by $G(T)$ in \eqref{sourceq} applies only during condensation, namely when $q \geq q_s$ and requires the upward motion ($\frac{dz}{dt}> 0$ or equivalently $\omega := \frac{dp}{dt}<0$ in pressure coordinates, see \cite{CZ} for instance). Since,
\begin{equation*}
    \frac{dz}{dt} = w = -\frac{H}{\pi} \dive v \quad \text{ then } \quad w> 0 \Longleftrightarrow \dive v < 0,
\end{equation*}
we can write \eqref{sourceq} in this way,
\begin{equation}\label{sourceq1}
   P = \delta \frac{Hg}{\pi R}(\dive v)^{-} vG(T),
\end{equation}
where $f^-=\max \{-f,0 \}$ denotes the negative part of $f$ and $\delta$ is defined as
\begin{equation*}
        \delta = \left \{ \begin{aligned} 
         1 \qquad & q \geq q_s \qquad \text{and } \dive v \leq 0\\
        0 \qquad & \text{ otherwise. }
        \end{aligned}\right.
    \end{equation*}
In this paper we make use of the equivalent formulation of \eqref{sourceq1} given by 
\begin{equation}\label{sourceq2}
    P = \frac{Hg}{\pi R}(\dive v)^{-} \mathcal{H}(q-q_s) G(T),
\end{equation}
where $\mathcal{H}(q-q_s)$ is the Heaviside multivalued function defined as,
\begin{equation}\label{heaviside}
        \mathcal{H}(r) = \left \{ \begin{aligned} 
         0 \qquad & r<0 \\
        (0,1) \qquad & r=0 \\
        1 \qquad &r>0
        \end{aligned}\right.
    \end{equation}
Moreover, we replace $G(T)$ by its positive part $G^+(T)$. As we will see later on, this assumption plays a central role in order to prove the existence and uniqueness of strong solutions for system \eqref{tro_model_fin}. Taking into account  the property $\eqref{prop_G}_{4}$ this requirement is linked with an upper bound of the temperature $T$,
\begin{equation}
\label{bound T}
T\leq \xi_0 = \frac {LR}{c_pR_\nu}.
\end{equation}

The condition mentioned above is not overly restrictive. In Section \ref{sec:6}, by combining the a priori estimates from Section \ref{sec:5} and the parabolic structure of the temperature balance equation, we can derive that $T$ is uniformly bounded in space and time (see \eqref{eq:finitet}). In particular, by applying the classical maximum principle, we can conclude that if we choose an initial datum satisfying \eqref{bound T}, then the same holds for $T$. We emphasize that the condition \eqref{bound T} will always be satisfied by an initial temperature since
$$ \xi_0 = \frac {LR}{c_pR_\nu}\simeq1548 \ \text{Kelvin},$$
which is never achievable for atmosphere temperatures.

Finally the system \eqref{tro_model} can be rewritten as follows

\begin{equation}\label{tro_model_fin}
     \left\{\begin{aligned}
&\partial_t u + u \cdot \nabla u - \Delta u + \nabla p + \dive (v \otimes v) = 0\\
& \partial_ t v + u \cdot \nabla v + v \nabla u - \Delta v = \frac{H}{\pi} \frac{g}{\theta_0} \nabla T \\
& \dive u = 0 \\
& \partial_t T + u \nabla T - \Delta T - \frac{H}{\pi}\frac{\theta_0 N^2}{g}  \dive v = \frac{H g}{\pi R}(\dive v)^- \mathcal{H}(q-q_s) G^+(T). \\
& \partial_t q + u \nabla q + \bar{Q} \dive v = -  \frac{H g}{\pi R}(\dive v)^- \mathcal{H}(q-q_s) G^+(T).
    \end{aligned}\right.
\end{equation}

\begin{remark}
    We point out that in order to deal with the discontinuous right hand side of the equations for $T$ and $q$ in  \eqref{tro_model_fin}, one can rephrase it as a differential inclusions (see \cite{CZ}). Therefore the system  \eqref{tro_model_fin} becomes,  
    \begin{equation}\label{tro_model_Te}
     \left\{\begin{aligned}
&\partial_t u + u \cdot \nabla u - \Delta u + \nabla p + \dive (v \otimes v) = 0\\
& \partial_ t v + u \cdot \nabla v + v \nabla u - \Delta v = \frac{H}{\pi} \frac{g}{\theta_0} \nabla T \\
& \dive u = 0 \\
& \partial_t T + u \nabla T - \Delta T -\frac{H}{\pi}\frac{\theta_0 N^2}{g} \dive v \in  \frac{H g}{\pi R}(\dive v)^- \mathcal{H}(q-q_s) G^+(T) \\
& \partial_t q + u \nabla q + \bar{Q} \dive v \in -  \frac{H g}{\pi R}(\dive v)^- \mathcal{H}(q-q_s) G^+(T).
    \end{aligned}\right.
\end{equation}
\end{remark}

\section{Main results}\label{sec:mainresults}
In this section, we will present the main results within the context of global strong solutions, which we define as follows.

\begin{definition}\label{def_strong}
Given a positive time $\mathcal{T}$ and initial data $(u_0 , v_0 , T_{0}, q_0) \in H^1(\mathbb{R}^2)$, then $(u,v,T,q)$ is called strong solution to system \eqref{tro_model_fin}  if it enjoys the following properties
\begin{equation*}
    \begin{aligned}
        & (u,v,T) \in C([0,\mathcal{T}]; H^1(\mathbb{R}^2)) \cap L^2((0,\mathcal{T}); H^2(\mathbb{R}^2));\\
        &(\partial_t u, \partial_t v, \partial_t T, \partial_t q) \in L^2((0,\mathcal{T}); L^2(\mathbb{R}^2)); \\
        &q \in C([0,\mathcal{T}]; L^2(\mathbb{R}^2)) \cap L^{\infty}((0,\mathcal{T}); H^1(\mathbb{R}^2)),
    \end{aligned}
\end{equation*}
has initial value
\begin{equation*}
    (u,v,T,q)\big|_{t=0} = (u_0 , v_0 , T_0, q_0)
\end{equation*} and satisfies the system \eqref{tro_model_fin} a.e on $\mathbb{R}^2 \times (0, \mathcal{T})$. More precisely, $T$ and $q$ satisfy a.e on $\mathbb{R}^2 \times (0, \mathcal{T})$ the following equations:
\begin{equation}\label{eq_T}
    \partial_t T + u \nabla T - \Delta T -\frac{H}{\pi}\frac{\theta_0 N^2}{g} \dive v =   \frac{H g}{\pi R}(\dive v)^- h_q G^+(T),
\end{equation}
\begin{equation}\label{eq_q}
    \partial_t q + u \nabla q + \bar{Q} \dive v = -  \frac{H g}{\pi R}(\dive v)^- h_q G^+(T),
\end{equation}
where $h_q \in L^{\infty}(\mathbb{R}^2 \times (0, \infty))$ is an element of $\mathcal{H}(q-q_s)$ defined as:
\begin{equation}\label{def_hq}
   h_q(x,y,t) := \left \{ \begin{aligned}
        & 1 \qquad q-q_s > 0 \\
        & \alpha \in [0,1] \qquad q-q_s=0 \\
        & 0 \qquad q-q_s < 0
    \end{aligned} \right.
\end{equation}
\end{definition}
\begin{remark}
If we adopt the differential inclusion formulation  as in \cite{CZ} (see system \eqref{tro_model_Te}), then \eqref{eq_q} and \eqref{def_hq} can be reformulated by saying that $h_q \in L^{\infty}_{t,x}$ and satisfies the following variational inequality,
\begin{equation}\label{def_var_ineq}
     \int_{\mathbb{R}^2}[\tilde{q}-q_s]^+dx - \int_{\mathbb{R}^2} [q-q_s]^+dx \geq  \int_{\mathbb{R}^2} h_q (\tilde{q}-q) dx \qquad \text{a.e in } [0,t_1] \times \mathbb{R}^2. 
\end{equation}
This means that $h_q$ is an element of the subdifferential of the function $q \rightarrow ([q-q_s]^+,1)$ where $(\cdot, 1)$ is the $L^2$ scalar product. Since it holds that
\begin{equation*}
    \partial([q-q_s]^+,1) = \mathcal{H}(q-q_s)
\end{equation*}
 it is easy to verify that if $h_q \in \mathcal{H}(q-q_s)$  then it satisfies \eqref{def_hq}.

In this paper, we will not use this formulation since, as we work on an unbounded domain, we will not obtain the compactness needed to prove \eqref{def_var_ineq}, as will be shown later.
\end{remark}

Now we are ready to state our main results.
\begin{theorem}(Global existence of strong solutions)\label{thm_final}
Suppose that $(u_0,v_0,T_{0},q_0) \in H^1(\mathbb{R}^2)$ with $\dive u_0=0$. Then there exist global strong solutions $(u,v,T,q)$ to system \eqref{tro_model_fin} in the sense of Definition \ref{def_strong} such that
\begin{equation}\label{ene_fin}
     \begin{aligned}
        \sup_{0 \leq t \leq \mathcal{T}}\|(u,v,T, q)\|_{H^1}^2 & + \int_0^\mathcal{T} (\|(\Delta u, \Delta v, \Delta T,\partial_t u, \partial_t v, \partial_t T, \partial_t q)\|_2^2 + \|(\nabla u, \nabla v)\|_{\infty})dt \leq C\\
    \end{aligned}
\end{equation}
where $h_q \in L^{\infty}(\mathbb{R}^2 \times (0, \infty))$ is an element of $\mathcal{H}(q-q_s)$ defined as \eqref{def_hq} and \\
$C:= C(\mathcal{T}, \|(u_0,v_0,T_{0},q_0)\|_{H^1}) $ is a nondecreasing function on $\mathcal{T} \in [0,\infty)$. 
\end{theorem}

\begin{theorem}(Uniqueness of global strong solutions)\label{teo:uniqueness}
Suppose that $(u_1, v_1, T_1, q_1)$ and $(u_2, v_2, T_2, q_2)$ are two global strong solutions to the system \eqref{tro_model_fin} as provided by Theorem \ref{thm_final}, with the same initial data and satisfying $(q_1 - q_s)(q_2 - q_s) \ge 0$ a.e in $(0,\mathcal{T}) \times \mathbb{R}^2$. Then the two solutions coincide; namely,
\begin{equation*}
(u_1, v_1, T_1, q_1) \equiv (u_2, v_2, T_2, q_2)
\quad \text{on } \mathbb{R}^2 \times [0,\mathcal{T}),
\quad \text{for every } \mathcal{T} \in [0,\infty).    
\end{equation*}
\end{theorem}
Concerning the uniqueness result, we note that the result is conditional to the assumption $(q_1 - q_s)(q_2 - q_s) \ge 0$. Physically, this assumption means that both solutions remain in the same thermodynamic regime: either below saturation ($q\leq q_s$) or above saturation ($q\geq q_s$). We stress that this requirement is physically sound. Indeed, since $q_s$ represents the saturation humidity, namely the maximal amount of water vapor sustainable before condensation processes become active, then the regions $q\leq q_s$ and $q\geq q_s$ correspond to distinct physical states of the atmosphere.
From the analytical viewpoint, the threshold $q=q_s$ is precisely where the nonlinear source term becomes discontinuous. Preventing changes of sign of $q-q_s$ allows one to control the multivalued contribution and derive a stability estimate leading to uniqueness. On the other hand, although it is natural to assume that the initial datum satisfies either $q_0\leq q_s$ or $q_0\geq q_s$, we are currently not able to prove that such a condition is propagated by the dynamics for positive times. Establishing the invariance of these regimes under the evolution remains an open issue, and for this reason the uniqueness statement is conditional.

\section{Approximate Problem}\label{sec:4}

In this section, we define a family of regularized problems starting from system~\eqref{tro_model_fin}. In particular, we introduce two different approximations. First, we regularize the Heaviside function by means of mollification, denoted by \(\mathcal{H}_\epsilon(q - q_s)\), to obtain a single-valued function. Next, we introduce an artificial viscosity term \(\eta \Delta q\) in the humidity equation. This strategy is designed to ensure the necessary regularity to prove local and, subsequently, global well-posedness for each fixed pair of parameters \(\epsilon, \eta > 0\).

\subsection*{I approximation: Problem $P_{\epsilon}$}
Let us introduce the real function $\mathcal{H}_\epsilon (\cdot)$ that approximate  $\mathcal{H}(\cdot)$,
\begin{equation}
\mathcal{H}_\epsilon (r) := \left \{ \begin{aligned}
             & 0, \quad r \leq 0 \\
            & r/\epsilon, \quad r \in (0, \epsilon] \\
            & 1, \quad r > \epsilon
        \end{aligned} \right.
\end{equation}
the function $\mathcal{H}_\epsilon(r)$ enjoys the following properties,
\begin{equation}\label{prop_H}
\begin{aligned}
        & |H_\epsilon(r_1)| \leq 1, \quad |H_\epsilon(r_1)-H_\epsilon(r_2)| \leq \frac{1}{\epsilon} |r_1-r_2| \qquad \text{for any}\  r_1,r_2 \in \mathbb{R}.
\end{aligned}
\end{equation}
Now, the approximating system $P_{\epsilon}$ is given by 
\begin{equation}\label{approx_epsilon}\tag{$P_{\epsilon}$}
\left\{\begin{aligned}
&\partial_t u + u \cdot \nabla u - \Delta u + \nabla p + \dive (v \otimes v) = 0\\
& \partial_ t v + u \cdot \nabla v + v \nabla u - \Delta v = \frac{H}{\pi} \frac{g}{\theta_0} \nabla T \\
& \dive u = 0 \\
& \partial_t T + u \nabla T - \Delta T - \frac{H}{\pi}\frac{\theta_0 N^2}{g}  \dive v = \frac{H g}{\pi R}(\dive v)^- \mathcal{H}_\epsilon(q-q_s) G^+(T) \\
& \partial_t q + u \nabla q + \bar{Q} \dive v = -  \frac{H g}{\pi R}(\dive v)^- \mathcal{H}_\epsilon(q-q_s) G^+(T), 
    \end{aligned}\right.
\end{equation}
such that for every fixed $\epsilon>0$ is a single-valued problem. In Section \ref{sec:5} we will prove the global well-posedness of system \eqref{approx_epsilon}. 

\subsection*{II approximation: Problem $P_{\epsilon,\eta}$}

In order to prove the existence of a global strong solution to \eqref{approx_epsilon}, one has to overcome the difficulty caused by the absence of diffusivity in the moisture equation. As already mentioned, we regularize \eqref{approx_epsilon} by adding an artificial viscosity term $\eta \Delta q$, with a positive parameter $\eta > 0$.

\begin{equation}\label{approx}\tag{$P_{\epsilon,\eta}$}
\left\{\begin{aligned}
&\partial_t u + u \cdot \nabla u - \Delta u + \nabla p + \dive (v \otimes v) = 0\\
& \partial_ t v + u \cdot \nabla v + v \nabla u - \Delta v = \frac{H}{\pi} \frac{g}{\theta_0} \nabla T \\
& \dive u = 0 \\
& \partial_t T + u \nabla T - \Delta T - \frac{H}{\pi}\frac{\theta_0 N^2}{g}  \dive v = \frac{H g}{\pi R}(\dive v)^- \mathcal{H}_\epsilon(q-q_s) G^+(T). \\
& \partial_t q + u \nabla q + \bar{Q} \dive v - \eta \Delta q = -  \frac{H g}{\pi R}(\dive v)^- \mathcal{H}_\epsilon(q-q_s) G^+(T). 
    \end{aligned}\right.
\end{equation}

For the system \eqref{approx} it is possible to prove the following global well-posedness result.
\begin{proposition}\label{global_existence_fixed_eta}
Suppose that $(u_0,v_0,T_{0},q_0) \in H^2(\mathbb{R}^2)$ such that $\dive u_0=0$. Then, there is a unique global strong solution $(u,v,T,q)$ to \eqref{approx} with initial data $(u_0,v_0,T_{0},q_0)$, such that
    \begin{equation*}
        \begin{aligned}
            & (u,v, T, q) \in C([0, \infty); H^2(\mathbb{R}^2))\cap L^2_{loc}([0, \infty); H^3(\mathbb{R}^2)), \\
            &(\partial_t u,  \partial_t v, \partial_t T, \partial_t q) \in L^2_{loc}([0, \infty); H^1(\mathbb{R}^2)).
        \end{aligned}
    \end{equation*}
In particular, we have for all $\mathcal{T} \in [0,\infty)$,
    \begin{equation}\label{uni_eta_energy}
    \begin{aligned}
        \sup_{0 \leq t \leq \mathcal{T}}\|(u,v,T, q)\|_{H^1}^2 & + \int_0^\mathcal{T} \|(\Delta u, \Delta v, \Delta T, \sqrt{\eta}\Delta q,  \partial_t u, \partial_t v, \partial_t T, \partial_t q)\|_2^2 dt \\
        & + \int_0^\mathcal{T}(\|\nabla u\|_{\infty} + \|\nabla v\|_\infty)dt \leq C
    \end{aligned}
    \end{equation}
    where $C:=C(\mathcal{T}, \|(u_0,v_0,T_{0},q_0)\|_{H^2}) $ is a nondecreasing function on $\mathcal{T} \in [0,\infty)$ and it is indipendent on $\eta$.
    
\end{proposition}

We would like to emphasize that the proof of Proposition \ref{global_existence_fixed_eta} is largely in line with those presented in \cite{LT, LT2}. The novelty of the present work lies in the inclusion of the diffusivity term in the temperature equation and in the specific form of the precipitation term. For this reason, we will focus only on analyzing the main differences related to these aspects.

\subsection*{Proof of Proposition \ref{global_existence_fixed_eta}}
The proof of Proposition \ref{global_existence_fixed_eta} consists of several steps. The local existence and uniqueness of strong solutions to system \eqref{approx} can be proved in a standard way, and here we only sketch the main arguments. For the sake of clarity, the $\eta$–uniform energy estimates are presented in Section \ref{sec:5}, since they are the key ingredients used to prove Theorem \ref{global_existence_thm_1}.

\subsubsection*{Contraction mapping principle}
We aim to apply the contraction mapping principle to the following linear system defined on $\mathbb{R}^2 \times (0,\mathcal{T})$ with initial data $(u_0,v_0,T_0,q_0)$:
\begin{equation*}
    \begin{aligned}
        & \partial_t U - \Delta U + \nabla P = - u \cdot \nabla u - \dive (v \otimes v) \\
        & \dive U=0\\
        &\partial_t V - \Delta V = - u \cdot \nabla v + \frac{H}{\pi} \frac{g}{\theta_0} \nabla T - v \cdot \nabla u \\
        &\partial_t \tilde{T} - \Delta \tilde{T} = - u \nabla T + \frac{H N^2 \theta_0}{\pi g}  \dive v + \frac{H g}{\pi R}(\dive v)^- \mathcal{H}_\epsilon(q-q_s) G^+(T) \\
        & \partial_t Q- \eta \Delta Q = - u \nabla q - \bar{Q} \dive v - \frac{H g}{\pi R}(\dive v)^- \mathcal{H}_\epsilon(q-q_s) G^+(T).
    \end{aligned}
\end{equation*}
Let us consider a positive time interval $\mathcal{T}$. For every $(u,v,T,q) \in L^2((0,\mathcal{T});H^3(\mathbb{R}^2)) \cap C([0,\mathcal{T}];H^2(\mathbb{R}^2))$, one can uniquely find a solution $(U,V,\tilde{T},Q) \in L^2((0,\mathcal{T});H^3(\mathbb{R}^2)) \cap C([0,\mathcal{T}];H^2(\mathbb{R}^2))$, with time derivatives $(\partial_t U, \partial_t V, \partial_t \tilde{T}, \partial_t Q) \in L^2((0,\mathcal{T});H^1(\mathbb{R}^2))$, to the linear system defined above, satisfying the initial condition $(u_0,v_0,T_0,q_0)$. By performing standard energy estimates, one can show that the solution map $\mathrm{R}$ defined as $$\mathrm{R}: (u,v,T,q) \rightarrow (U,V,\tilde{T},Q)$$ is a contraction in the space $L^2((0,\mathcal{T});H^3(\mathbb{R}^2)) \cap C([0,\mathcal{T}];H^2(\mathbb{R}^2))$, for sufficiently small positive time $\mathcal{T}$ depending only on the initial data $(u_0,v_0,T_0,q_0)$. By the Banach fixed-point theorem, this yields the existence of a unique fixed point $(u,v,T,q)$, corresponding to the unique local strong solution of the system with the prescribed initial data.

\subsubsection*{Global existence} 
In order to prove global existence, one needs to show that the maximal existence time $\mathcal{T}^*>0$ of the unique strong solution to \eqref{approx} satisfies $\mathcal{T}^*= \infty$. We proceed by contradiction assuming that $\mathcal{T}^* < \infty$. Applying the operator $\nabla$ to equation $\eqref{approx}_5$ and taking the $L^2(\mathbb{R}^2)$ inner product with $- \nabla \Delta q$, we obtain
\begin{equation*}
\begin{aligned}
   & \frac{1}{2} \frac{d}{dt} \|\Delta q\|_2^2 + \|\sqrt{\eta} \nabla \Delta q\|_2^2   \leq C \int_{\mathbb{R}^2} [|u| |\nabla^2 q| + |\nabla^2 v| + |\nabla v \nabla T| + |\nabla v \nabla q|] |\nabla \Delta q|dx \\
   & \leq C ( \|u\|_4 \|\nabla^2 q\|_4 + \|\Delta v\|_2 + \|\nabla v\|_4 \|\nabla T\|_4 + \|\nabla v\|_4 \|\nabla q\|_4  ) \|\nabla \Delta q\|_2 \\
   & \leq C \Big( 
    \|u\|_2^{1/2} \|\nabla u\|_2^{1/2} \|\Delta q\|_2^{1/2} \|\nabla \Delta q\|_2^{1/2} 
    + \|\Delta v\|_2 \\
&\quad + \|\nabla v\|_2^{1/2} \|\Delta v\|_2^{1/2} 
      \|(\nabla q, \nabla T)\|_2^{1/2} \|(\Delta q, \Delta T)\|_2^{1/2} 
\Big) \|\nabla \Delta q\|_2 \\
& \leq \frac{1}{4} \|\sqrt{\eta}\, \nabla \Delta q\|_2^2 
    + C_{\eta} \Big( 
        \|u\|_2 \|\nabla u\|_2 \|\Delta q\|_2 \|\nabla \Delta q\|_2 
        + \|\Delta v\|_2^2 \\
& \quad + \|\nabla v\|_2 \|\Delta v\|_2 
        \|(\nabla q, \nabla T)\|_2 \|(\Delta q, \Delta T)\|_2
      \Big) \\
   & \leq \frac{1}{2}  \|\sqrt{\eta}\nabla \Delta q\|^2_2 + C_{\eta} ( 1 + \|u\|_2^2 \|\nabla u\|_2^2 + \|(\nabla v, \nabla q, \nabla T)\|_2^2)\|)\|(\Delta v, \Delta q, \Delta T)\|_2^2.
\end{aligned}
\end{equation*}
By repeating the same argument for $(u,v,T)$ as in \cite{LT2}, one can obtain the following estimate
\begin{equation*} 
\begin{aligned} 
& \sup_{t \in [0, \mathcal{T}^*]} \|(u,v,T,q)\|_{H^2} + \int_0^{\mathcal{T}^*} \|(\nabla u, \nabla v,\nabla T, \nabla q)\|_{H^2} dt \leq C, 
\end{aligned} 
\end{equation*}
where $C := C(\eta, \|(u_0,v_0,T_0,q_0)\|_{H^2})$. Therefore, one can extend the solution beyond the time $\mathcal{T}^*$, which contradicts the hypothesis. This contradiction implies that $\mathcal{T}^*= \infty$, and consequently we obtain a global strong solution.

\section{Global existence of strong solutions to \eqref{approx_epsilon} for every fixed $\epsilon>0$}\label{sec:5}

This section is devoted to prove the following existence result of global strong solutions for the system \eqref{approx_epsilon}. 

\begin{theorem}\label{global_existence_thm_1}
   Suppose that $(u_0,v_0,T_{0},q_0) \in H^1(\mathbb{R}^2)$ 
   with $\dive u_0=0$. Then there exist a unique global strong solution $(u,v,T,q)$ to system \eqref{approx_epsilon} for every fixed $\epsilon>0$ such that
   \begin{equation*}
       \begin{aligned}
           & (u,v,T) \in C([0,\mathcal{T}];H^1(\mathbb{R}^2)) \cap L^2((0,\mathcal{T});H^2(\mathbb{R}^2)) \\
           &q \in C([0,\mathcal{T}];L^2(\mathbb{R}^2)) \cap L^\infty((0,\mathcal{T});H^1(\mathbb{R}^2)) \\
           & (\partial_t u, \partial_t v, \partial_t T, \partial_t q) \in L^2((0,\mathcal{T}); L^2(\mathbb{R}^2))\\
       \end{aligned}
   \end{equation*}
   for any positive time $\mathcal{T}>0$. Moreover it holds for every $\epsilon >0$
\begin{equation}\label{ene_uni_epsilon}
     \begin{aligned}
        \sup_{0 \leq t \leq \mathcal{T}}\|(u,v,T, q)\|_{H^1}^2 & + \int_0^\mathcal{T} (\|(\Delta u, \Delta v, \Delta T, \partial_t u, \partial_t v, \partial_t T, \partial_t q)\|_2^2dt\\
        &+ \|\nabla u\|_{\infty} + \|\nabla v\|_{\infty})dt \leq C,
    \end{aligned}
\end{equation}
 where $C:=C(\mathcal{T}, \|(u_0,v_0,T_{0},q_0)\|_{H^1})$ is a nondecreasing function on $\mathcal{T} \in [0,\infty)$ and it is indipendent on $\epsilon$.
\end{theorem}

In order to prove Theorem \ref{global_existence_thm_1}, it is necessary to consider the limit $\eta \to 0$ in the system \eqref{approx}. As a first step, we compute uniform estimates in $\eta$ satisfied by the unique strong solution $(u,v,T,q)$ of system \eqref{approx} with initial data $(u_0,v_0,T_0,q_0) \in H^2(\mathbb{R}^2)$.

\begin{proposition}(Basic energy estimate)\label{1ene}
 Let $(u,v,T,q)$ be the unique global strong solution to system \eqref{approx}, with initial data $(u_0,v_0,T_{0},q_0) \in H^2(\mathbb{R}^2)$. Then it holds
\begin{equation}\label{energy_approx}
    \begin{aligned}
        \sup_{t \in [0, \mathcal{T})} & \left(\|u\|_2^2 + \|v\|_2^2 +  \|T\|_2^2 + \|q\|_2^2 \right)  \\
        & + \int_0^{\mathcal{T}} 2\|\nabla u\|_2^2 + \|\nabla v\|_2^2 + \|\nabla T\|_2^2 + 2\|\sqrt{\eta} \nabla q\|_2^2 dt \leq C, 
    \end{aligned}
\end{equation}
where $C$ is a constant indipendent of $\eta, \epsilon$.
\end{proposition}

\begin{proof}
We start by multiplying equation $\eqref{approx}_1$ by $u$ and $\eqref{approx}_2$ by $v$, integrating over $\mathbb{R}^2$ we get
\begin{equation}\label{eneu}
         \frac{1}{2} \frac{d}{dt} \left( \|u\|^2_2 + \|v\|^2_2  \right) +  \|\nabla u\|^2_2 +  \|\nabla v\|^2_2 
        - \frac{H}{\pi} \frac{g}{\theta_0} \int_{\mathbb{R}^2} \nabla T  \cdot v dx = 0.
\end{equation}
Then we multiply $\eqref{approx}_4$ by $T$ and integrate over $\mathbb{R}^2$,
\begin{equation}\label{enete}
    \frac{1}{2} \frac{d}{dt} \|T\|^2_2 +  \|\nabla T\|_2^2 =  \frac{H \theta_0 N^2}{\pi g} \int_{\mathbb{R}^2} T \dive v dx+ \frac{Hg}{\pi R}\int_{\mathbb{R}^2} (\dive v)^- G^+(T)\mathcal{H}_\epsilon(q-q_s) T dx.
\end{equation}
Finally, we multiply $\eqref{approx}_5$ by $q$ and we obtain
\begin{equation}\label{eneq}
\begin{aligned}
    & \frac{1}{2} \frac{d}{dt} \|q\|_2^2 + \eta \|\nabla q\|_2^2 = - \bar{Q} \int_{\mathbb{R}^2} q \cdot \dive v dx - \frac{Hg}{\pi R}\int_{\mathbb{R}^2}(\dive v)^- \mathcal{H}_\epsilon(q-q_s) G^+(T)  \cdot q \; dx.
\end{aligned}
\end{equation}
We define $E(t)$ as:
\begin{equation*}
    E(t):= \frac{1}{2} \left(\|u\|_2^2 + \|v\|_2^2 +  \|T\|_2^2 + \|q\|_2^2 \right)(t),
\end{equation*}
Summing up $\eqref{eneu}, \eqref{enete}$ and $\eqref{eneq}$ one obtains 
\begin{equation}\label{der_ene}
    \begin{aligned}
        \frac{d}{dt} & E(t)  + \|\nabla u\|_2^2 + \|\nabla v\|_2^2 +  \|\nabla T\|_2^2 + \eta \|\nabla q\|_2^2 = \frac{H}{\pi} \frac{g}{\theta_0} \int_{\mathbb{R}^2} \nabla T  \cdot v dx \\
        &  +  \frac{H \theta_0 N^2}{\pi g} \int_{\mathbb{R}^2} T \dive v dx + \frac{Hg}{\pi R}\int_{\mathbb{R}^2} (\dive v)^- G^+(T)\mathcal{H}_\epsilon(q-q_s) T dx \\
        &- \bar{Q} \int_{\mathbb{R}^2} q \cdot \dive v dx - \frac{Hg}{\pi R}\int_{\mathbb{R}^2}(\dive v)^- \mathcal{H}_\epsilon(q-q_s) G^+(T)  \cdot q \; dx.
    \end{aligned}
\end{equation}

Since $\left| \mathcal{H}_\epsilon (q-q_s)\right|\leq 1$, by using \eqref{prop_G} we get
\begin{equation}\label{gw}
    \begin{aligned}
        &\bar{Q} \int_{\mathbb{R}^2} \left| q \cdot \dive v \right| dx +   \frac{Hg}{\pi R} \int_{\mathbb{R}^2} \left| (\dive v)^-\mathcal{H}_{\epsilon}(q-q_s) G^+(T)  \cdot q \right|\; dx \\
        & + \frac{H \theta_0 N^2}{\pi g} \int_{\mathbb{R}^2} |T \dive v| dx+ \frac{Hg}{\pi R}\int_{\mathbb{R}^2} \left|(\dive v)^- G^+(T)\mathcal{H}_\epsilon(q-q_s) T\right| dx \\
        & \leq   \frac{1}{2} \|\nabla v\|_2^2 +  \Tilde{C} \|(T,q)\|_2^2
    \end{aligned}
\end{equation}
and 
\begin{equation}\label{5.7}
    \frac{H}{\pi} \frac{g}{\theta_0} \int_{\mathbb{R}^2} |\nabla T  \cdot v |dx \leq \frac{1}{2} \|\nabla T\|_2^2 + \tilde{C}\|v\|_2^2,
\end{equation}
where $\tilde{C}= \tilde{C}(H, \pi,g,\theta_0, \bar{Q}, C_G)$, $C_G$ defined in \eqref{prop_G}. Combining  \eqref{der_ene} with \eqref{gw} by standard application of Gronwall lemma we conclude that
\begin{equation*}
    \begin{aligned}
        \sup_{t \in (0, \mathcal{T})} & \left(\|u\|_2^2 + \|v\|_2^2 +  \|T\|_2^2 + \|q\|_2^2 \right)  \\
        & + \int_0^{\mathcal{T}} 2\|\nabla u\|_2^2 + \|\nabla v\|_2^2 + \|\nabla T\|_2^2 + 2\|\sqrt{\eta} \nabla q\|_2^2 dt \leq C, 
    \end{aligned}
\end{equation*}
where $C=C(H,\pi,g,\theta_0,C_G, \|u_0,v_0,T_{0},q_0\|_{H^2},\mathcal{T})$ is indipendent of $\epsilon$ and $\eta$. 
\end{proof}
At this point, we establish the $L^{\infty}((0,\mathcal{T}); H^1(\mathbb{R}^2))$ regularity for $(u,v,T)$.

\begin{proposition}\label{prop2}
    The following estimate holds:
    \begin{equation}\label{3ene}
        \sup_{t \in[0,\mathcal{T})} \|(\nabla u, \nabla v, \nabla T)\|_2^2 + \int_0^\mathcal{T}\|\Delta u, \Delta v, \Delta T\|_2^2 dt \leq C,
    \end{equation}
    where C is a positive constant $C= C(\mathcal{T},\|(u_0,v_0,q_0,T_{0}\|_{H^2})$ indipendent on $\epsilon$ and $\eta$. 
\end{proposition}

\begin{proof}
We begin by multiplying the equations 
$\eqref{approx}_1,\eqref{approx}_2, \eqref{approx}_4$ 
by $(-\Delta u, -\Delta v, -\Delta T)$, respectively.
Integrating by parts, we obtain

    \begin{equation}\label{gradu}
    \begin{aligned}
        \frac{1}{2}\frac{d}{dt}\|\nabla u\|_2^2 + \|\Delta u\|_2^2 & = \int_{\mathbb{R}^2}(u\nabla u + \dive(v \otimes v))\Delta u dx \\
        & \leq C(\|u\|_4\|\nabla u\|_4 + \|v\|_4\|\nabla v\|_4)\|\Delta u\|_2 \\
        & \leq C(\|u\|_4^2\|\nabla u\|_4^2 + \|v\|_4^2 \|\nabla v\|_4^2)+ \frac{1}{8}\|\Delta u\|_2^2 \\
        & \leq C( \|u\|_4^2\|\nabla u\|_2\|\Delta u\|_2 + \|v\|_4^2\|\nabla v\|_2\|\Delta v\|_2) + \frac{1}{8}\|\Delta u\|_2^2 \\
        &\leq C(\|u\|_4^4\|\nabla u\|_2^2 + \|v\|_4^4 \|\nabla v\|_2^2) + \frac{1}{4}\|\Delta u\|_2^2 + \frac{1}{8}\|\Delta v\|_2^2,
    \end{aligned}
\end{equation}
    thanks to Young and Ladyzhenskaya inequalities. Similarly for $v$ we get
     \begin{equation}\label{gradv}
    \begin{aligned}
        \frac{1}{2}\frac{d}{dt}\|\nabla v\|_2^2 + \|\Delta v\|_2^2 \leq C\|(u,v)\|_4^4 \|(\nabla u,\nabla v)\|_2^2 + \frac{3}{8}\|\Delta v\|_2^2 + \frac{1}{4}\|\Delta u\|_2^2 + C \|\nabla T\|_2^2.
    \end{aligned}
\end{equation}
Concerning the equation for $T$ we obtain,
\begin{equation}\label{gradqt}
    \begin{aligned}
        \frac{1}{2}\frac{d}{dt}\|\nabla T\|_2^2 + \|\Delta T\|_2^2 \leq C\|u\|_4^4 \|\nabla T\|_2^2 + \frac{1}{2}\|\Delta T\|_2^2 + C\|\nabla v\|_2^2.
    \end{aligned}
\end{equation}
Thanks to Proposition \ref{1ene} we note that
\begin{equation*}
    \int_0^\mathcal{T} \|u\|_4^4 dt \leq \int_0^\mathcal{T} \|u\|_2^2 \|\nabla u\|_2^2 dt \leq \sup_{ t \in [0,\mathcal{T}]} \|u\|_2^2 \int_0^\mathcal{T} \|\nabla u\|_2^2 dt \leq  C,
\end{equation*}
that holds also for $v$ and $T$. Therefore, by summing \eqref{gradu}, \eqref{gradv} and \eqref{gradqt}, we recover \eqref{3ene} by applying Gronwall's lemma.

\end{proof}

\begin{proposition}\label{grad_infty}
     The following estimate holds:
    \begin{equation}\label{grad_inf}
        \int_0^\mathcal{T}\|(\nabla u, \nabla v)\|_\infty dt \leq C 
    \end{equation}
    where C is a positive constant $C= C(\mathcal{T},\|(u_0,v_0,q_0,T_{0}\|_{H^2})$ indipendent on $\epsilon$ and $\eta$.
\end{proposition}
\begin{proof}
    We start by showing that $\nabla v \in L^1((0,\mathcal{T}); L^\infty(\mathbb{R}^2))$. To prove this, we aim to apply Lemma \ref{A3}; in particular, we need to check that
\begin{equation*}
    f(x,t) := - u \nabla v - v \nabla u + \nabla T \in L^2((0,\mathcal{T}); L^4(\mathbb{R}^2)),
\end{equation*}
in order to conclude that $\Delta v \in L^1((0,\mathcal{T}); L^4(\mathbb{R}^2))$.

    First, we note that $\nabla T \in L^2((0,\mathcal{T}; L^4(\mathbb{R}^2))$ and 
    \begin{equation}\label{grad_4}
        \begin{aligned}
            & \int_0^\mathcal{T} \|(\nabla u,\nabla v)\|_4^4 dt \leq \int_0^\mathcal{T}\|(\nabla u, \nabla v)\|_2^2 \|(\Delta u, \Delta v)\|_2^2 dt \leq C
        \end{aligned}
    \end{equation}
    thanks to Proposition \ref{prop2}. We now proceed to estimate the term $u \nabla v$ as follows:
\begin{equation*}
\begin{aligned}
   \int_0^\mathcal{T}\|u \nabla v\|_4^2 dt \leq & \int_0^\mathcal{T}\|u\|_\infty^2 \|\nabla v\|_4^2 dt \leq \left( \int_0^\mathcal{T}\|u\|_\infty^4 dt\right)^{1/2} \left( \int_0^\mathcal{T}\|\nabla v\|_4^4 dt\right)^{1/2} \\
   & \leq C\left( \int_0^\mathcal{T}\|u\|_2^2 \|\Delta u\|_2^2 dt\right)^{1/2} \left( \int_0^\mathcal{T}\|\nabla v\|_4^4 dt\right)^{1/2} \\
   & \leq C \left( \sup_{t \in [0,\mathcal{T}]} \|u\|_2^2\int_0^\mathcal{T}\|\Delta u\|_2^2 dt\right)^{1/2} \left( \int_0^\mathcal{T}\|\nabla v\|_4^4 dt\right)^{1/2} \leq C
\end{aligned}
\end{equation*}
thanks to \eqref{grad_4} and Propositions \ref{1ene} and \ref{prop2}. Similarly, it can be shown that $v \nabla u \in L^2((0,\mathcal{T}); L^4(\mathbb{R}^2))$. At this stage, Lemma \ref{A3} can be applied to conclude that

\begin{equation}\label{lapla_v_4}
    \int_0^\mathcal{T}\|\Delta v\|_4 dt \leq C,
\end{equation}
where $C$ is a generic constant indipendent from $\epsilon,\eta$. By using \eqref{lapla_v_4}, Proposition \ref{1ene}, Proposition \ref{prop2} and Gagliardo Niremberg Sobolev inequality we deduce that
\begin{equation}\label{grad_v_infty}
    \begin{aligned}
        \int_0^\mathcal{T}\|\nabla v\|_\infty dt & \leq \int_0^\mathcal{T} \|v\|_\infty^{\frac{1}{3}}\|\Delta v\|_4^{\frac{2}{3}}dt \leq \left(\int_0^\mathcal{T} \|v\|_\infty dt \right)^{\frac{1}{3}} \left( \int_0^\mathcal{T} \|\Delta v\|_4dt \right)^{\frac{2}{3}}\!\!\leq  C,
    \end{aligned}
\end{equation}
where C is a positive constant $C= C(\mathcal{T},\|(u_0,v_0,q_0,T_{0}\|_{H^2})$ indipendent on $\epsilon$ and $\eta$. 

Following the approach of Proposition 3.5 in \cite{LT}, we establish that 
\[
\int_0^\mathcal{T} \|\nabla u\|_{\infty} \, dt
\]
is bounded by a constant independent of both $\epsilon$ and $\eta$. In particular, we decompose $u= \bar{u}+ \hat{u}$ such that $\bar{u}$ and $\hat{u}$ are the unique solutions to the following systems:
\begin{equation}
    \left\{ \begin{aligned}\label{baru}
        & \partial_t \bar{u} - \Delta \bar{u} + \nabla \bar{p}= - u \nabla u - \dive (v \otimes v) \\
        & \dive \bar{u} = 0\\
        & \bar{u}|_{t=0} = 0
    \end{aligned} \right.
\end{equation}
and 
\begin{equation}\label{hatu}
    \left\{ \begin{aligned}
        & \partial_t \hat{u} - \Delta \hat{u} + \nabla \hat{p}= 0\\
        & \dive \hat{u}= 0\\
        &\hat{u}|_{t=0} = u_0.
    \end{aligned} \right.
\end{equation}
We start by  estimating  $\bar{u}$. Denote $U=(u,v)$, by the $L^q(0,\mathcal{T}; W^{2,q})$ type estimates for the Stokes equations we have (see \cite{LT}),
\begin{equation*}
    \|(\partial_t \bar{u}, \Delta \bar{u}\|_{L^q( \mathbb{R}^2 \times (0,\mathcal{T})} \leq C \|U \nabla U\|_{L^q( \mathbb{R}^2 \times (0,\mathcal{T})}, 
\end{equation*}
for any $q \in (0, \infty)$. Therefore,
\begin{equation*}
\begin{aligned}
    &  \int_0^\mathcal{T} \|\Delta \bar{u}\|_4 dt \leq \int_0^\mathcal{T} \|U \nabla U\|_4 dt \leq C \int_0^\mathcal{T} \|\nabla U\|_4 \|U\|_{\infty} dt\\
    & \leq C \left( \int_0^\mathcal{T} \|U\|_\infty^4 dt \right)^{\frac{1}{4}} \left( \int_0^\mathcal{T} \|\nabla U\|_4^4  dt\right)^{\frac{1}{4}} \sqrt{\mathcal{T}} \leq C
\end{aligned}
\end{equation*}
thanks to \eqref{grad_4} and Propositions \ref{1ene}-\ref{prop2}. As previously observed for $v$, by applying the Gagliardo–Nirenberg–Sobolev inequality and invoking Propositions \ref{1ene} and \ref{prop2} for $\bar u$, we obtain

\begin{equation}\label{grad_u_bar_infty}
    \begin{aligned}
        \int_0^\mathcal{T}\|\nabla \bar{u}\|_\infty dt & \leq \int_0^\mathcal{T} \|\bar u\|_\infty^{\frac{1}{3}}\|\Delta \bar u \|_4^{\frac{2}{3}}dt \leq \left(\int_0^\mathcal{T} \|\bar u\|_\infty dt \right)^{\frac{1}{3}} \left( \int_0^\mathcal{T} \|\Delta \bar u\|_4dt \right)^{\frac{2}{3}} \leq  C. 
    \end{aligned}
\end{equation}
Now we estimate $\hat{u}$, multiplying equation \eqref{hatu} by $(t \Delta^2 - \Delta)\hat{u}$ and integrating over $\mathbb{R}^2$ we get
\begin{equation*}
    \frac{1}{2} \frac{d}{dt} ( \|\nabla \hat{u}\|_2^2 + \|\sqrt{t} \Delta \hat{u}\|_2^2) + \frac{1}{2} \|\Delta \hat{u}\|_2^2 + \|\sqrt{t} \nabla \Delta \hat{u}\|_2^2 = 0,
\end{equation*}
and therefore
\begin{equation*}
    \sup_{t \in [0,\mathcal{T})} ( \|\nabla \hat{u}\|_2^2 + \|\sqrt{t} \Delta \hat{u}\|_2^2) + \int_0^\mathcal{T} \|\Delta \hat{u}\|_2^2 + \|\sqrt{t} \nabla \Delta \hat{u}\|_2^2dt \leq \|\nabla u_0\|_2^2.
\end{equation*}
Since $\|\phi\|_\infty \leq C \|\phi\|_2^{\frac{1}{2}} \|\Delta \phi\|_2^{\frac{1}{2}}$ we have
\begin{equation*}
\begin{aligned}
    \int_0^\mathcal{T} \|\nabla \hat{u}\|_\infty dt & \leq C \int_0^\mathcal{T} \|\nabla \hat{u}\|_2^{\frac{1}{2}} \|\nabla \Delta \hat{u}\|_2^{\frac{1}{2}}dt \\
    & = C \int_0^\mathcal{T} \|\nabla \hat{u}\|_2^{\frac{1}{2}} \|\sqrt{t}\nabla \Delta \hat{u}\|_2^{\frac{1}{2}}t^{-\frac{1}{4}}dt \\
    & \leq C \left( \int_0^\mathcal{T} \|\nabla \hat{u}\|_2^2 \right)^{\frac{1}{4}}\left( \int_0^\mathcal{T} \|\sqrt{t}\nabla \Delta \hat{u}\|_2^2\right)^{\frac{1}{4}}\mathcal{T}^{\frac{1}{4}} \\
    &\leq C \|\nabla u_0\|_2^{\frac{1}{2}} \|\nabla u_0\|_2^{\frac{1}{2}} \mathcal{T}^{\frac{1}{4}}.
\end{aligned}
\end{equation*}
Combining the above estimates we get:
\begin{equation}\label{control_nablau}
    \int_0^\mathcal{T} \|\nabla u\|_\infty dt \leq \int_0^\mathcal{T} \|\nabla \bar u \|_\infty  + \|\nabla \hat u\|_\infty dt \leq C, 
\end{equation}
where $C=C(\mathcal{T}, \|(u_0,v_0,q_0,T_{0}\|_{H^2})$. Now, the proof is concluded. 
\end{proof}

At this stage, we are ready to show that $q \in L^{\infty}((0,\mathcal{T}); H^1(\mathbb{R}^2))$ as well.

\begin{proposition} \label{h1_q}
    Let $(u,v,T,q)$ be the unique global strong solution to system \eqref{approx} with initial data $(u_0,v_0,T_{0},q_0) \in H^2(\mathbb{R}^2)$. Then we have the following estimate, for all $\mathcal{T} \in [0,\infty)$:
    \begin{equation*}
    \sup_{t \in (0,\mathcal{T})} \|\nabla q\|_2^2  + \eta \int_0^{\mathcal{T}}\|\Delta q \|_2^2 dt \leq C
    \end{equation*}
    where $C:= C(\mathcal{T},\|(u_0,v_0,T_0,q_0)\|_{H^2})$ is a nondecreasing continuous function on  $\mathcal{T} \in [0,\infty)$ and it is indipendent of $\eta$.
    \end{proposition}
\begin{proof}
We multiply the equation $\eqref{approx}_5$ by $-\Delta q$, integrating by parts we get
\begin{equation*}
\begin{aligned}
    \frac{1}{2}\frac{d}{dt}\|\nabla q\|_2^2 + \eta \|\Delta q\|_2^2  & =  - \int_{\mathbb{R}^2}\nabla u |\nabla q|^2 dx + \bar{Q} \int_{\mathbb{R}^2} \dive v \Delta q dx \\
    & + \frac{Hg}{\pi R} \int_{\mathbb{R}^2}  (\dive v)^- G^+(T)\mathcal{H}_\epsilon(q-q_s) \Delta q \\
    & = - \int_{\mathbb{R}^2}\nabla u |\nabla q|^2 dx - \bar{Q} \int_{\mathbb{R}^2} \nabla \dive v \nabla q dx \\
    & - \frac{Hg}{\pi R} \int_{\mathbb{R}^2} \nabla (\dive v)^- G^+(T)\mathcal{H}_\epsilon(q-q_s) \nabla qdx \\
    & - \frac{Hg}{\pi R} \int_{\mathbb{R}^2} (\dive v)^- G'(T)\cdot \mathbf{1}_{\{G(T) > 0\}} \nabla T \nabla qdx \\
    & - \frac{Hg}{\pi R} \int_{\mathbb{R}^2} (\dive v)^- G^+(T) \mathcal{H}_\epsilon'(q-q_s)|\nabla q|^2dx.
\end{aligned}
\end{equation*}
Thanks to Propositions \ref{1ene}-\ref{grad_infty} and the properties \eqref{prop_H} of $\mathcal{H}_\epsilon$ it holds
\begin{equation*}
\begin{aligned}
    & \frac{1}{2}\frac{d}{dt}\|\nabla q\|_2^2 + \eta \|\Delta q\|_2^2 +  \frac{Hg}{\pi R} \int_{\mathbb{R}^2} (\dive v)^- G^+(T) \mathcal{H}_\epsilon'(q-q_s)|\nabla q|^2dx \\
    & \leq  C(\|\nabla u\|_\infty + 1) \|\nabla q\|_2^2 + C(\|\nabla v\|_4^4\|\nabla T\|_2^2 + \|\Delta v\|_2^2 + \|\Delta T\|_2^2).
\end{aligned}  
\end{equation*}
   Finally by the standard application of Gronwall lemma we obtain
    \begin{equation*}
         \sup_{t \in [0,\mathcal{T})} \|\nabla q\|_2^2  + \eta \int_0^{\mathcal{T}}\|\Delta q \|_2^2dt \leq C,
    \end{equation*}
    where $C:= C(\mathcal{T},\|(u_0,v_0,T_0,q_0)\|_{H^2})$ is a nondecreasing continuous function on \\  $\mathcal{T} \in [0,\infty)$ and it is indipendent of $\eta$.
\end{proof}

At this stage we are able to prove that the temperature $T$ remains bounded almost everywhere in $(x,t)$. This result rules out the occurrence of arbitrarily large temperature values, in agreement with physical observations in tropical regions.

\begin{proposition}\label{max_princ}
 Let $(u,v,T,q)$ be the unique global strong solution to system \eqref{approx}, with initial data $(u_0,v_0,T_{0},q_0) \in H^2(\mathbb{R}^2)$. Then it holds for all $\mathcal{T} \in [0,\infty)$
    \begin{equation*}
    \begin{aligned}
        &  \sup_{t \in [0, \mathcal{T})}\|T\|_{L^\infty(\mathbb{R}^2)} \leq C,
    \end{aligned}
    \end{equation*}
    where $C:= C(\mathcal{T},\|(u_0,v_0,T_{0},q_0)\|_{H^2})$ is a nondecreasing function on $\mathcal{T} \in [0,\infty)$ and it is indipendent on $\eta,\epsilon$.
\end{proposition}
\begin{proof}
We recall that the equation satisfied by $T$ is given by
\begin{equation*}
    \partial_t T + u \nabla T - \Delta T - \frac{H}{\pi}\frac{\theta_0 N^2}{g}  \dive v = \frac{H g}{\pi R}(\dive v)^- \mathcal{H}_\epsilon(q-q_s) G^+(T),
\end{equation*}
therefore by the standard application of the Duhamel principle we know that the solution can be expressed in the following way
\begin{equation}
    \begin{aligned}
        & T(x,t)  = \int_{\mathbb{R}^2}K(t,x-y)T_0(y)dy - \int_0^t\!\!\int_{\mathbb{R}^2}\!\!K(t-s,x-y)u\nabla Tdyds\\
        & + \int_0^t\!\!\int_{\mathbb{R}^2}\!\!K(t-s,x-y)\left(\frac{H}{\pi}\frac{\theta_0 N^2}{g}  \dive v + \frac{H g}{\pi R}(\dive v)^-  \mathcal{H}_\epsilon(q-q_s) G^+(T) \right)(s,y)dyds,
    \end{aligned}
\end{equation}
where the convolution kernel in $\mathbb{R}^2$ is given by
\begin{equation*}
    K(x,t) = \frac{1}{4\pi t}e^{-\frac{|x|^2}{4t}}.
\end{equation*}
Now we consider the absolute value of $T$, 
\begin{equation*}
    \begin{aligned}
        & |T(x,t)|  \leq \int_{\mathbb{R}^2}|K(t,x-y)T_0(y)|dy +  \int_0^t\!\!\int_{\mathbb{R}^2}\left| K(t-s,x-y) u\nabla T \right|dyds\\
        & + \int_0^t\!\!\int_{\mathbb{R}^2}\left| K(t-s,x-y)\left( \frac{H}{\pi}\frac{\theta_0 N^2}{g}  \dive v + \frac{H g}{\pi R}(\dive v)^- \mathcal{H}_\epsilon(q-q_s) G^+(T) \right)(s,y)\right|dyds \\
        & \; = I_1 + I_2+ I_3.
    \end{aligned}
\end{equation*}
Regarding  $I_1$ we have
\begin{equation*}
    |I_1| \leq \|T_0\|_{\infty}.
\end{equation*}
For the other terms $I_2$ and $I_3$ we apply Young’s inequality for convolutions. Recall that, for any function \( F(x,t) \in L^4(\mathbb{R}^2) \), the following estimate holds:
\begin{equation*}
    \|[K(t-s) * F(s)](x)\|_\infty \leq \|K(t-s)\|_{4/3} \|F(s)\|_4.
\end{equation*}
Moreover, from Propositions \ref{1ene} - \ref{h1_q},  \eqref{prop_G} and \eqref{prop_H} we know that
\begin{equation*}
    \begin{aligned}
        & (T,u,v) \in L^{\infty}((0,\mathcal{T});H^1(\mathbb{R}^2))\cap L^2((0,\mathcal{T});H^2(\mathbb{R}^2)), \\
        & \nabla v \in L^1((0,\mathcal{T}),L^\infty(\mathbb{R}^2)),\\
        & G^+(T) \leq C_G,\\
        &|\mathcal{H}_\epsilon(q-q_s)| \leq 1
    \end{aligned}
\end{equation*}
and thus, it holds 
\begin{equation*}
    \begin{aligned}
        & \|u \nabla T\|_{L^2((0,\mathcal{T});L^4(\mathbb{R}^2))} \leq C(\mathcal{T}, \|(u_0,v_0,T_0,q_0)\|_{H^2(\mathbb{R}^2)}) \\
        & \|\dive v\|_{L^1((0,\mathcal{T}),L^\infty(\mathbb{R}^2))}\leq C(\mathcal{T}, \|(u_0,v_0,T_0,q_0)\|_{H^2(\mathbb{R}^2)}).
    \end{aligned}
\end{equation*}
One can control $I_2$ in the following way
\begin{equation*}
    \begin{aligned}
         \|[K(t-s) * (u \nabla T)(s)](x)\|_\infty & \leq  \|K(t-s)\|_{4/3} \|(u \nabla T)(s)\|_4, \\
         & \sim \frac{1}{(t-s)^{\frac{1}{4}}} \|(u \nabla T)(s)\|_4,
    \end{aligned}
\end{equation*}
thus
\begin{equation*}
\begin{aligned}
     \int_0^t \frac{1}{(t-s)^{\frac{1}{4}}} \|(u \nabla T)(s)\|_4 ds & \leq C  \left( \int_0^t \frac{1}{(t-s)^{\frac{1}{2}}}ds \right)^{\frac{1}{2}} \cdot \left( \int_0^t \|(u \nabla T)(s)\|_4^2 ds ds \right)^{\frac{1}{2}} \\
     & \leq C t^{\frac{1}{4}}.
\end{aligned}
\end{equation*}
The other terms can be treated similarly and in particular,
\begin{equation*}
    \begin{aligned}
        & \|[K(t-s) * (\dive v)(s)](x)\|_\infty \leq \|K(t-s)\|_{1} \|\dive v (s)\|_\infty \sim \|\nabla v (s)\|_\infty.
    \end{aligned}
\end{equation*}
therefore
\begin{equation*}
    \begin{aligned}
        \int_0^t \|[K(t-s) * (\dive v)(s)](x)\|_\infty ds \leq \int_0^t \|\nabla v\|_\infty ds \leq C(t, \|(u_0,v_0,T_0,q_0)\|_{H^1}).
    \end{aligned}
\end{equation*}
Finally, by taking the supremum over $\mathbb{R}^2 \times [0,\mathcal{T})$ we deduce

\begin{equation*}
\begin{aligned}
& \sup_{t \in [0,\mathcal{T})} \sup_{x \in \mathbb{R}^2} |T(x,t)| 
\leq \sup_{t \in [0,\mathcal{T})} \|T_0\|_{\infty} \\
&+ \sup_{t \in [0,\mathcal{T})} \int_0^t \sup_{x \in \mathbb{R}^2} \Bigg| 
          \int_{\mathbb{R}^2} K(t-s,x-y) ( u \nabla T)(s,y) dy \Bigg|ds\\
& \quad + \sup_{t \in [0,\mathcal{T})} \int_0^t 
      \sup_{x \in \mathbb{R}^2} \Bigg| 
          \int_{\mathbb{R}^2} K(t-s,x-y) \\
& \hspace{2em} \times \Big(\frac{H}{\pi} \frac{\theta_0 N^2}{g} \dive v 
           + \frac{H g}{\pi R} (\dive v)^- \mathcal{H}_\epsilon(q-q_s) G^+(T) \Bigg)(s,y) \, dy 
      \Bigg| \, ds \\
& \leq \|T_0\|_{\infty} 
      +  C \int_0^{\mathcal{T}}
        \Big\| K(t-s) * \Big(u \nabla T + \dive v \Big)(s) \Big\|_{\infty} \, ds \\
& \leq C( 1 + \|T_0\|_{\infty} +  \mathcal{T}^{\frac{1}{4}}) \leq C\Big(\mathcal{T}, \|(u_0,v_0,T_0,q_0)\|_{H^2(\mathbb{R}^2)}\Big).
\end{aligned}
\end{equation*}
The proof is concluded.
\end{proof}
\begin{remark}
In contrast to \cite{CZ}, we are not able to prove that
\( 0 \leq T \leq M \) for all \( (x,t) \in \mathbb{R}^2 \times [0,\mathcal{T}] \). 
More precisely, even by following the arguments developed in \cite{CZ2}, 
we can not show that the negative part of \( T \), denoted by \( T^- \), 
vanishes almost everywhere in \( \mathbb{R}^2 \times [0,\mathcal{T}] \). 
This difficulty arises from the fact that, in the present analysis, we  consider an 
unbounded spatial domain, and that the model involves 
\( G(T)^+ \) rather than \( F(T)^+ \), as in the previous works \cite{CZ, CZ2}.
\end{remark}

On the basis of Propositions \ref{1ene}–\ref{max_princ}, we can summarize the following a priori estimates, which are independent of $\eta > 0$.

\begin{corollary}\label{a_priori_estimate_eta}
    Let $(u,v,T,q)$ be the unique global strong solution to system \eqref{approx}, with initial data $(u_0,v_0,T_{0},q_0) \in H^2(\mathbb{R}^2)$ and consider $\eta \in (0, 1)$ . Then we have for all $\mathcal{T} \in [0,\infty)$
    \begin{equation*}
    \begin{aligned}
        &  \sup_{t \in [0, \mathcal{T}]}\left( \|(u,v,T, q)\|_{H^1}^2 + \|T\|_\infty \right) + \int_0^\mathcal{T} (\|(\Delta u, \Delta v, \Delta T, \sqrt{\eta}\Delta q\|_2^2 + \|\nabla u\|_{\infty} + \|\nabla v\|_{\infty})dt  \\
        & + \int_0^\mathcal{T} \|\partial_t u, \partial_t v, \partial_t T, \partial_t q)\|_2^2dt  \leq C,
    \end{aligned}
    \end{equation*}
    where $C:= C(\mathcal{T},\|(u_0,v_0,T_{0},q_0)\|_{H^2})$ is a nondecreasing function on $\mathcal{T} \in [0,\infty)$ and it is indipendent on $\eta$.
\end{corollary}
\begin{proof}
It remains to verify only the part involving the time derivatives. Concerning the moisture equation $\eqref{approx}_5$, we have 
    \begin{equation*}
    \begin{aligned}
        \int_0^\mathcal{T} \|\partial_t q\|_2^2 dt & \leq \int_0^\mathcal{T} (\|u \nabla q\|_2^2 + \eta\|\sqrt{\eta} \Delta q\|_2^2 + (\bar{Q}^2 + C_G^2)\|\dive v\|_2^2) dt \\
        &\leq C \sup_{t \in [0,\mathcal{T})} \|\nabla q\|_2^2 \int_0^\mathcal{T} \|u\|_2 \|\Delta u\|_2 dt + C(\eta + 1) \\
        & \leq C \left(\int_0^\mathcal{T} \|u\|_2^2\right)^{1/2} \left( \int_0^\mathcal{T} \|\Delta u\|_2^2 \right)^{1/2} + 2C \leq \tilde{C},
    \end{aligned}
    \end{equation*}
    where $\tilde{C}:= \tilde{C}(\|(u_0,v_0,T_{0},q_0)\|_{H^2}, \mathcal{T})$.
One can repeat the same argument for $(u,v,T)$ to conclude the proof.
\end{proof}

\subsection{Proof of Theorem \ref{global_existence_thm_1}}\label{subsec_4.1}

As a final step we perform the limit as $\eta \rightarrow 0$ in \eqref{approx}. We consider a sequence of initial data $(u_{0,\eta}, v_{0,\eta}, T_{0,\eta}, q_{0,\eta}) \in H^2(\mathbb{R}^2)$ with $\nabla \cdot u_{0,\eta} = 0$, such that,

\begin{equation*}
    \begin{aligned}
        & (u_{0,\eta},v_{0,\eta}, T_{0,\eta}, q_{0,\eta}) \rightarrow (u_{0},v_{0}, T_{0}, q_{0}) \; \text{ in } H^1(\mathbb{R}^2) \text{ as } \eta \rightarrow 0, \\
        &\|(u_{0,\eta},v_{0,\eta}, T_{0,\eta}, q_{0,\eta})\|_{H^1} \leq \|(u_{0},v_{0}, T_{0}, q_{0})\|_{H^1},\\
        &\|T_{0,\eta}\|_\infty \leq \|T_0\|_\infty.
    \end{aligned}
\end{equation*}
By virtue of Proposition \ref{global_existence_fixed_eta} and Corollary \ref{a_priori_estimate_eta}, for each $\eta \in (0,1)$ and initial data  $(u_{0,\eta}, v_{0,\eta}, T_{0,\eta}, q_{0,\eta})$, we know there exists a unique global strong solution $(u_\eta, v_\eta, T_\eta, q_\eta)$ to the system \eqref{approx} satisfying,

\begin{equation*}
\begin{aligned}
    & \sup_{t \in [0,\mathcal{T}]} ( \|(u_\eta,v_\eta,T_{\eta},q_\eta)\|^2_{H^1} + \|T_{\eta}\|_\infty)  \\
    & + \int_0^\mathcal{T}\|(\Delta u_\eta, \Delta v_\eta, \Delta T_{\eta}, \sqrt{\eta}\Delta q_\eta\|_2^2 + \|\nabla u_\eta \|_{\infty} + \|\nabla v_\eta \|_{\infty}dt \\
    & + \int_0^\mathcal{T} \| \partial_t u_\eta, \partial_t v_\eta, \partial_t T_{\eta}, \partial_t q_\eta)\|_2^2 \; dt \leq C, \\
\end{aligned}
\end{equation*}
for a positive constant $C$ independent of $\eta$.
Thanks to the above estimates, we can extract a subsequence still depends on $\mathcal{T}$, denoted by $(u_\eta, v_\eta, T_{\eta},q_\eta)$ such that
\begin{equation}\label{conv_1}
    \begin{aligned}
        & (u_\eta, v_\eta, T_{\eta},q_\eta) \overset{\ast}{\rightharpoonup} (u,v,T,q) \; \text{ in } L^{\infty}((0,\mathcal{T});H^1(\mathbb{R}^2)); \\
        & (\Delta T_{\eta}, \sqrt{\eta}\Delta q_\eta) \overset{}{\rightharpoonup} (\Delta T,g) \; \text{ in } L^{2}((0,\mathcal{T});L^2(\mathbb{R}^2));\\
        & (u_\eta, v_\eta) \overset{}{\rightharpoonup} (u,v) \; \text{ in } L^{2}((0,\mathcal{T});H^2(\mathbb{R}^2));\\
        &(\partial_t u_\eta, \partial_tv_\eta, \partial_tT_{\eta},\partial_tq_\eta) \overset{}{\rightharpoonup} (\partial_t u, \partial_tv, \partial_tT,\partial_tq) \; \text{ in } L^{2}((0,\mathcal{T});L^2(\mathbb{R}^2)).\\
    \end{aligned}
\end{equation}
By the Aubin–Lions lemma (Corollary 4, \cite{simons}), for every positive integer $k$, there exists a subsequence $(u_\eta, v_\eta, T_\eta, q_\eta)$ such that:

\begin{equation}\label{reg_sol_eta}
    \begin{aligned}
        &(u_{\eta},v_{\eta}, T_{\eta}) \rightarrow (u,v,T) \text{ in } C([0,\mathcal{T}];H^1(B_k)) , \\
        & q_{\eta} \rightarrow q \text{ in } C([0,\mathcal{T}];L^2(B_k));
    \end{aligned}
\end{equation}
where $B_k$ is the disk of radius $k$. By applying the Cantor diagonal argument first with respect to $\eta$ and $k$, and subsequently with respect to $\eta$ and $\mathcal{T}$, one can show that the above convergence still holds for every $R>0$ and $\mathcal{T}>0$. Owing to the continuity and boundedness of $G(T)$ and $H_\epsilon(r)$ for each fixed $\epsilon > 0$ (see \eqref{prop_G} and \eqref{prop_H}), and to the strong convergence of $(u_\eta, v_\eta, T_\eta, q_\eta)$ established above, we conclude that, up to subsequences:
\begin{equation*}
\begin{aligned}
    & H_\epsilon( q_\eta - q_s) \rightarrow H_\epsilon(q-q_s) \qquad \text{a.e in } \mathbb{R}^2 \times (0, \mathcal{T}), \\
    &  \left(\frac{F(T_\eta)}{T_\eta}\right)^+ = G^+(T_\eta) \rightarrow G^+(T)= \left(  \frac{F(T)}{T} \right)^+  \qquad \text{a.e in } \mathbb{R}^2 \times (0, \mathcal{T}), \\
    & \nabla u_\eta \rightarrow \nabla u \qquad \text{a.e in } \mathbb{R}^2 \times (0, \mathcal{T}),\\
    &\nabla v_\eta  \rightarrow \nabla v \qquad \text{a.e in } \mathbb{R}^2 \times (0, \mathcal{T}),\\
    &T_\eta \rightarrow T \qquad \text{a.e in } \mathbb{R}^2 \times (0, \mathcal{T}).
\end{aligned}
\end{equation*}
Therefore, the limit solution satisfies the equations \eqref{approx_epsilon} for almost every $(x,t)$. Moreover, by the Fatou lemma, we deduce that $\nabla u, \nabla v \in L^1_t L^\infty_x$. Finally, thanks to the weak lower semicontinuity of the norm and \eqref{conv_1}, we conclude that the limit solution satisfies the energy inequality \eqref{ene_uni_epsilon} for every $\epsilon > 0$,
\begin{equation*}
     \begin{aligned}
        \sup_{t \in [0,\mathcal{T})}\left( \|(u,v,T, q)\|_{H^1}^2 + \|T\|_\infty \right) & + \int_0^\mathcal{T} \|(\Delta u, \Delta v, \Delta T, \partial_t u, \partial_t v, \partial_t T, \partial_t q)\|_2^2 dt\\
        & + \int_0^\mathcal{T} \|\nabla u\|_{\infty} + \|\nabla v\|_{\infty}dt  \leq C,
    \end{aligned}
\end{equation*}
 where $C:= C(\mathcal{T}, \|(u_0,v_0,T_{0},q_0)\|_{H^1})$ is a nondecreasing function on $ \mathcal{T} \in [0,\infty)$ and it is indipendent on $\epsilon$.

\subsection{Uniqueness of strong solution for system \eqref{approx_epsilon}}\label{unique_1}
The proof of uniqueness of strong solutions for the system \eqref{approx_epsilon} is standard. Let
$$
(\delta u, \delta v, \delta T, \delta q) := (u_1, v_1, T_1, q_1) - (u_2, v_2, T_2, q_2)
$$
denote the difference between two strong solutions with the same initial data. By performing the classical $L^2$-type energy estimates and subsequently applying a logarithmic Gronwall inequality, the desired result follows. The equations satisfied by $(\delta u, \delta v, \delta T, \delta q)$ are given by

\begin{equation*}
\left\{
    \begin{aligned}
     &\partial_t \delta u + u_1 \nabla \delta u + \delta u \nabla u_2 - \Delta \delta u + \nabla \delta p + \dive (v_1 \otimes \delta v) + \dive (\delta v \otimes v_2) = 0 \\
    &\partial_t \delta v + u_1 \nabla \delta v + \delta u \nabla v_2 - \Delta \delta v + v_1 \nabla \delta u + \delta v \nabla u_2 = \frac{H}{\pi} \frac{g}{\theta_0} \nabla \delta T \\
    &\dive \delta u  = 0 \\
    &\partial_t \delta T + u_1 \nabla \delta T + \delta u \nabla T_{2} - \Delta \delta T - \frac{H}{\pi}\frac{N^2 \theta_0}{g} \dive  \delta v  = \frac{H g}{\pi R}(\dive v_1)^- \mathcal{H}_\epsilon(q_1-q_s) G(T_1)^+  \\
    & \hspace{8.5cm}- \frac{H g}{\pi R} (\dive v_2)^- \mathcal{H}_\epsilon(q_2-q_s) G(T_2)^+ \\
    &\partial_t \delta q + u_1 \nabla \delta q + \delta u \nabla q_2 + \bar{Q} \dive \delta v  = -  \frac{H g}{\pi R}(\dive v_1)^- \mathcal{H}_\epsilon(q_1-q_s)G(T_1)^+ \\
    & \hspace{6cm}+ \frac{H g}{\pi R} (\dive v_2)^- \mathcal{H}_\epsilon(q_2-q_s) G(T_2)^+. 
\end{aligned}\right.
\end{equation*}
We multiply the first two equations by $(\delta u, \delta v)$ and integrate over $\mathbb{R}^2$. After some rearrangements we get,
\begin{equation}\label{pt_1_1}
    \begin{aligned}
        & \frac{1}{2} \frac{d}{dt} \left( \|\delta u\|_2^2 + \|\delta v\|_2^2  \right) + \frac{1}{2}\|\nabla \delta u\|_2^2 + \frac{3}{4}\|\nabla \delta v\|_2^2 \\
        & \leq C \int_{\mathbb{R}^2}[(|\nabla u_2|+ |\nabla v_2| + |\nabla v_1| + |v_1|^2 + |v_2|^2] (|\delta u|^2 + |\delta v|^2)  dx + C \|\delta T\|_2^2.
    \end{aligned}
\end{equation}
Concerning the temperature equation we have,
\begin{equation*}
    \begin{aligned}
        & \frac{1}{2}\frac{d}{dt} \|\delta T\|_2^2+ \|\nabla \delta T\|_2^2 = - \int_{\mathbb{R}^2} \delta u \nabla T_2 \delta T dx + \frac{H}{\pi}\frac{N^2 \theta_0}{g}\int_{\mathbb{R}^2} \dive \delta v \delta T dx \\
        & + \int_{\mathbb{R}^2} \left[ \frac{H g}{\pi R}(\dive v_1)^-\mathcal{H}_\epsilon(q_1-q_s) G^+(T_1) - \frac{H g}{\pi R}(\dive v_2)^- \mathcal{H}_\epsilon(q_2-q_s) G^+(T_2) \right] \delta qdx \\
        & := I_1' +I_2'.
    \end{aligned}
\end{equation*}
We control $I_1'$ in the following way,
\begin{equation*}
    |I_1| \leq C \int_{\mathbb{R}^2} [|\delta u| |\nabla T_2| + |\nabla \delta v|]|\delta T | dx \leq \frac{1}{16} \|\nabla \delta v\|_2^2 + C\|\delta T\|_2^2 + \|\delta u\|_\infty \|\nabla T_2\|_2 \|\delta T\|_2.
\end{equation*}
Regarding $I_2'$ we have
\begin{equation*}
    \begin{aligned}
        |I_2'| &= \left|\int_{\mathbb{R}^2}(\dive v_1)^- - (\dive v_2)^-)G^+(T_1) \mathcal{H_\epsilon}(q_1-q_s) \delta T dx\right| \\
        & + \left|\int_{\mathbb{R}^2}(\dive v_2)^- (G^+(T_1)-G^+(T_2))\mathcal{H}_\epsilon(q_1-q_s) \delta T dx\right| \\
        & + \left|\int_{\mathbb{R}^2} (\dive v_2)^- G^+(T_2)(\mathcal{H}_\epsilon(q_1 -q_s) - \mathcal{H}_\epsilon(q_2 -q_s))\delta T\right|\\
        & \leq \frac{1}{16} \|\nabla \delta v\|_2^2 + C( 1 + \|\nabla v_2\|_2^2)\|\delta T\|_2^2 + \frac{1}{4}\|\nabla \delta T\|_2^2 + + C\frac{1}{\epsilon}\int_{\mathbb{R}^2}|\delta q\|\delta T|dx \\
        & \leq \frac{1}{16} \|\nabla \delta v\|_2^2 + C_\epsilon( 1 + \|\nabla v_2\|_2^2)\|(\delta T, \delta q)\|_2^2 + \frac{1}{4}\|\nabla \delta T\|_2^2 ,
    \end{aligned}
\end{equation*}
where we have used the properties \eqref{prop_H} of $\mathcal{H_\epsilon}$. We can conclude that
\begin{equation}\label{pt_1_2}
    \begin{aligned}
        \frac{1}{2}\frac{d}{dt} \|\delta T\|_2^2+ \frac{3}{4}\|\nabla \delta T\|_2^2 & \leq \frac{1}{8} \|\nabla \delta v\|_2^2 + C_\epsilon( 1 + \|\nabla v_2\|_2^2)\|(\delta T, \delta q)\|_2^2 \\
        & \; + \|\delta u\|_\infty \|\nabla T_2\|_2 \|\delta T\|_2.
    \end{aligned}
\end{equation}
Finally, we multiply by $\delta q$ and integrate over $\mathbb{R}^2$ to obtain
\begin{equation}\label{est_q}
    \begin{aligned}
        & \frac{1}{2}\frac{d}{dt} \|\delta q\|_2^2 = - \int_{\mathbb{R}^2}  (u_1 \nabla \delta q + \delta u \nabla q_2 + \bar{Q} \dive \delta v) \delta q dx \\
        & - \int_{\mathbb{R}^2} \left[ \frac{H g}{\pi R}(\dive v_1)^-\mathcal{H}_\epsilon(q_1-q_s) G^+(T_1) - \frac{H g}{\pi R}(\dive v_2)^- \mathcal{H}_\epsilon(q_2-q_s) G^+(T_2) \right] \delta qdx \\
        & := I_1+ I_2
    \end{aligned}
\end{equation}
Concerning $I_1$, as before we have
\begin{equation*}
    \begin{aligned}
        |I_1| & \leq \|\delta u\|_\infty \|\nabla q_2\|_2\|\delta q\|_2 + \frac{1}{16} \|\nabla \delta v\|_2^2 + \|\delta q\|_2^2,
    \end{aligned}
\end{equation*}
while for $I_2$ we get
\begin{equation*}
    \begin{aligned}
        I_2 & = - \int_{\mathbb{R}^2} [(\dive v_1)^- -(\dive v_2)^-]G^+(T_1) \mathcal{H}_\epsilon(q_1 -q_s) \delta qdx \\
        & -  \int_{\mathbb{R}^2} (\dive v_2)^-(G^+(T_1)-G^+(T_2))\mathcal{H}_\epsilon(q_2 -q_s)\delta q dx \\
        &- \int_{\mathbb{R}^2} (\dive v_2)^- G^+(T_1)(\mathcal{H}_\epsilon(q_1 -q_s) - \mathcal{H}_\epsilon(q_2 -q_s))\delta q dx,
    \end{aligned}
\end{equation*}
we recall that $(\mathcal{H}_\epsilon(q_1 -q_s) - \mathcal{H}_\epsilon(q_2 -q_s))\delta q \geq 0$ thanks to the monotonicity properties \eqref{prop_H}. By using  Holder, Ladyzhenskaya and Young inequalities we get
\begin{equation}\label{pt_1_3}
    \begin{aligned}
        \frac{1}{2}\frac{d}{dt} \|\delta q\|_2^2 & + \int_{\mathbb{R}^2} (\dive v_2)^- G^+(T_1)(\mathcal{H}_\epsilon(q_1 -q_s) - \mathcal{H}_\epsilon(q_2 -q_s))\delta q dx  \leq \|\delta u\|_\infty \|\nabla q_2\|_2\|\delta q\|_2 \\
        & + \frac{1}{8} \|\nabla \delta v\|_2^2 + C( \|\nabla v_2\|_4^4 \|\delta T\|_2^2 + \|\delta q\|_2^2 ) + \frac{1}{4} \|\nabla \delta T\|_2^2.
    \end{aligned}
\end{equation}
Combining \eqref{pt_1_1}, \eqref{pt_1_2} and \eqref{pt_1_3} we obtain
\begin{equation*}
    \begin{aligned}
        & \frac{d}{dt} \left( \|\delta u\|_2^2 + \|\delta v\|_2^2 + \|\delta T\|_2^2 + \|\delta q\|_2^2 \right) + \frac{1}{2} (\|\nabla \delta u\|_2^2 + \|\nabla \delta v\|_2^2 + \|\nabla \delta T\|_2^2) \\
        & \leq C \int_{\mathbb{R}^2}[(|\nabla u_2|+ |\nabla v_2| + |\nabla v_1| + |v_1|^2 + |v_2|^2] (|\delta u|^2 + |\delta v|^2)  dx \\
       &+ C_\epsilon(1 + \|\nabla v_2\|_4^4 + \|\nabla v_2\|_2^2)\|(\delta T,\delta q)\|_2^2 + \|\delta u\|_\infty \|(\nabla T_2, \nabla q_2)\|_2\|(\delta T,\delta q)\|_2.
    \end{aligned}
\end{equation*}
By using once again Holder, Ladyzhenskaya and Young inequalities we deduce
\begin{equation*}
    \begin{aligned}
        & \frac{d}{dt} \left( \|\delta u\|_2^2 + \|\delta v\|_2^2 + \|\delta T\|_2^2 + \|\delta q\|_2^2 \right) + \frac{1}{2} (\|\nabla \delta u\|_2^2 + \|\nabla \delta v\|_2^2 + \|\nabla \delta T\|_2^2) \\
        & \leq C \left[\|(|\nabla u_2|, |\nabla v_2|,|\nabla v_1|)\|_2  + \|(v_1,v_2)\|_4^2 \right]\|(\delta u, \delta v)\|_4^2 \\
        &+ C_\epsilon(1 + \|\nabla v_2\|_4^4 + \|\nabla v_2\|_2^2)\|(\delta T,\delta q)\|_2^2 + \|\delta u\|_\infty \|(\nabla T_2, \nabla q_2)\|_2\|(\delta T,\delta q)\|_2\\
        & \leq C \left[\|(|\nabla u_2|, |\nabla v_2|,|\nabla v_1|)\|_2  + \|(v_1,v_2)\|_4^2 \right] \|(\delta u, \delta v)\|_2 \|(\nabla \delta u, \nabla \delta v)\|_2  \\
        &+ C_\epsilon(1 + \|\nabla v_2\|_4^4 + \|\nabla v_2\|_2^2)\|(\delta T,\delta q)\|_2^2+ \|\delta u\|_\infty \|(\nabla T_2, \nabla q_2)\|_2\|(\delta T,\delta q)\|_2 \\
        & \leq \frac{1}{4}\|(\nabla \delta u, \nabla \delta v)\|_2^2 + C \|(\nabla T_{2}, \nabla q_2)\|_2 \|\delta u, \delta v, \delta T \|_{L^{\infty}} \|(\delta T,\delta q)\|_2 \\
        & + C_\epsilon \left[ 1+ \|(|\nabla u_2|, |\nabla v_2|,|\nabla v_1|)\|_2^2 +  \|(v_1,v_2)\|_4^4 + \|\nabla v_2\|_4^4 \right] \|(\delta u, \delta v, \delta T,\delta q)\|_2^2.
    \end{aligned}
\end{equation*}
Therefore, one has
\begin{equation*}
    \begin{aligned}
        & \frac{d}{dt} \left( \|\delta u\|_2^2 + \|\delta v\|_2^2 + \|\delta T\|_2^2 + \|\delta q\|_2^2 \right) + \frac{1}{4} (\|\nabla \delta u\|_2^2 + \|\nabla \delta v\|_2^2 + \|\nabla \delta T\|_2^2) \\
        & \leq C_\epsilon( 1 +  \|(v_1,v_2)\|_4^4 + \|(|\nabla u_2|, |\nabla v_2|,|\nabla v_1|)\|_2^2 + \|\nabla v_2\|_4^4) \|(\delta u, \delta v, \delta T, \delta q)\|_2^2 \\
        &+  C \|(\nabla T_{2}, \nabla q_2)\|_2 \|(\delta u, \delta v,\delta T)\|_{L^{\infty}} \|(\delta T,\delta q)\|_2.
    \end{aligned}
\end{equation*}
Now, recalling the Brezis-Gallouet-Wainger inequality 
\begin{equation*}
    \|f\|_{L^{\infty}(\mathbb{R}^2)} \leq C \|f\|_{H^1(\mathbb{R}^2)}\log^{\frac{1}{2}}\left( \frac{\|f\|_{H^2(\mathbb{R}^2)}}{\|f\|_{H^1(\mathbb{R}^2)}} + e\right),
\end{equation*}
and denoting $U_1=(u_1,v_1,T_1)$, $U_2=(u_2,v_2,T_2)$ and $\delta U = (\delta u, \delta v,\delta T)$ we have:
\begin{equation*}
\begin{aligned}
    \|\delta U\|_{L^{\infty}(\mathbb{R}^2)} & \leq C \|\delta U\|_{H^1(\mathbb{R}^2)}\log^{\frac{1}{2}}\left( \frac{\|\delta U\|_{H^2(\mathbb{R}^2)}}{\|\delta U\|_{H^1(\mathbb{R}^2)}} + e\right) \\
    & \leq  C \|\delta U\|_{H^1(\mathbb{R}^2)}\log^{\frac{1}{2}}\left( \frac{S(t)}{\|\delta U\|_{H^1(\mathbb{R}^2)}} \right)\\
    & = C \left[ \|\delta U\|^2_{H^1(\mathbb{R}^2)} \log^{+}\frac{S(t)}{\|\delta U\|_{H^1(\mathbb{R}^2)}} \right]^{\frac{1}{2}},
\end{aligned}
\end{equation*}
where $$S(t) := \|U_1\|_{H^2(\mathbb{R}^2)} + \|U_2\|_{H^2(\mathbb{R}^2)} + e\left(\|U_1\|_{H^1(\mathbb{R}^2)}+ \|U_2\|_{H^1(\mathbb{R}^2)}\right).$$
Let us define
\begin{equation*}
    \begin{aligned}
        & f = \|\delta u\|_2^2 + \|\delta v\|_2^2 + \|\delta T\|_2^2 + \|\delta q\|_2^2, \\
        & G= \frac{1}{4}\left( \|\nabla \delta u\|_2^2 + \|\nabla \delta v\|_2^2 + \|\nabla \delta T\|_2^2 \right), \\
        &m_1=  1 +  \|(v_1,v_2)\|_4^4 + \|(|\nabla u_2|, |\nabla v_2|,|\nabla v_1|)\|_2^2 + \|\nabla v_2\|_4^4, \\
        & m_2= C\|(\nabla T_{2}, \nabla q_2)\|_2, 
    \end{aligned}
\end{equation*}
it follows from the previous relations that:
\begin{equation*}
    f' + G \leq C_\epsilon m_1 f + m_2\left[ f G \log^+ \left(\frac{S}{4G} \right)\right]^{\frac{1}{2}}.
\end{equation*}
At this point, we can apply Lemma \ref{lemma_uniqueness} to conclude that $f \equiv 0$, which proves uniqueness for every fixed $\epsilon > 0$. This completes the proof of Theorem \ref{global_existence_thm_1}.

\section{Global existence and uniqueness of strong solutions to the Tropical Climate Model \eqref{tro_model_fin}}\label{sec:6}

In this section we provide the proof of Theorem \ref{thm_final}. In particular, Subsection \ref{existence} is devoted to establishing the existence of solutions via the classical compactness-based approach, together with the pointwise convergence of the moisture and temperature equations \eqref{eq_T} and \eqref{eq_q}. Then, in Subsection \ref{uniqueness}, we address the uniqueness of solutions by adapting the argument presented in Section 4 of \cite{LT}.

\subsection{Proof of Theorem \ref{thm_final} (Global Existence)}\label{existence}
We recall that the global in time strong solution of \eqref{approx_epsilon} satisfies, for every fixed $\epsilon > 0$, the following bounds uniform in $\epsilon$ derived from the energy inequality \eqref{ene_uni_epsilon},
\begin{equation}\label{uniboundeps}
    \begin{aligned}
        & \|u_{\epsilon}\|_{L^{\infty}_t H^1_x} \leq C & \|u_{\epsilon}\|_{L^{2}_t H^2_x} \leq C \quad  & \|v_{\epsilon}\|_{L^{\infty}_t H^1_x} \leq C  & \|v_{\epsilon}\|_{L^{2}_t H^2_x} \leq C \\
        &\|T_{\epsilon}\|_{L^{\infty}_t H^1_x} \leq C & \|T_{\epsilon}\|_{L^{2}_t H^2_x} \leq C  \quad  & \| \nabla u_{\epsilon} \|_{L^{1}_t L^{\infty}_x} \leq C & \|\partial_t u_{\epsilon}\|_{L^{2}_t L^{2}_x} \leq C \\
        & \|\partial_t v_{\epsilon}\|_{L^{2}_t L^{2}_x} \leq C \quad & \|\partial_t T_{\epsilon}\|_{L^{2}_t L^{2}_x} \leq C \quad & \|\partial_t q_{\epsilon}\|_{L^{2}_t L^{2}_x} \leq C \quad & \|q_{\epsilon}\|_{L^{\infty}_t H^1_x} \leq C. \\
    \end{aligned}
\end{equation}
In addition, as a consequence of Proposition \ref{max_princ}, we also have that
\begin{equation}\label{eq:linftyt}
\|T_{\epsilon}\|_{L^{\infty}_t L^{\infty}_x}\leq C,
\end{equation}
with $C>0$ independent on $\epsilon$. 
Therefore, there exists a subsequence (not relabelled) of $(u_\epsilon,v_\epsilon,q_\epsilon, T_{ \epsilon})$ such that,
\begin{equation}\label{conv_eps}
\begin{aligned}
    & (u_{\epsilon},v_{\epsilon}, T_{\epsilon},q_{\epsilon}) \overset{\ast}{\rightharpoonup} (u,v,T,q) \text{ in } L^{\infty}((0,\mathcal{T});H^1( \mathbb{R}^2)), \\
    & (u_{\epsilon},v_{\epsilon},T_{\epsilon}) \overset{}{\rightharpoonup} (u,v,T) \text{ in } L^{2}((0,\mathcal{T});H^2( \mathbb{R}^2)),\\
    &  (\partial_t u_{\epsilon}, \partial_t v_{\epsilon}, \partial_t T_{\epsilon}, \partial_t q_{\epsilon}) \overset{}{\rightharpoonup} (\partial_t u, \partial_t v, \partial_t T, \partial_t q) \text{ in } L^{2}((0,\mathcal{T});L^2( \mathbb{R}^2)),\\
    & \mathcal{H}_\epsilon(q_\epsilon - q_s) \overset{\ast}{\rightharpoonup} h_q \text{ in } L^{\infty}(\mathbb{R}^2 \times (0,\mathcal{T})),
\end{aligned}
\end{equation}
where the last convergence comes from the $L^{\infty}_{t,x}$ bound of the function $\mathcal{H}_\epsilon$. In addition, we also have that 
\begin{equation*}
T_{\epsilon}\overset{\ast}{\rightharpoonup} T\mbox{ in }L^{\infty}(\mathbb{R}^2 \times (0,\mathcal{T})),
\end{equation*}
and thus the limiting temperature satisfies 
\begin{equation}\label{eq:finitet}
T\in L^{\infty}(\mathbb{R}^2 \times (0,\mathcal{T})). 
\end{equation}

Applying Aubin-Lions lemma, for any positive time $\mathcal{T}$ and a disk $B_R \subseteq \mathbb{R}^2$ of arbitrary radius $R>0$ it holds:
\begin{equation}\label{a_l_eps}
    \begin{aligned}
        & (u_{\epsilon},v_{\epsilon}, T_{\epsilon}) \rightarrow (u,v,T) \text{ in } C([0,\mathcal{T});L^2(B_R))\cap L^2((0,T);H^1(B_R)) , \\
        & q_{\epsilon} \rightarrow q \text{ in } C([0,\mathcal{T});L^2(B_R)).
    \end{aligned}
\end{equation}
As argued in the previous section, the convergences in \eqref{conv_eps}, together with Fatou's Lemma and the weak lower semicontinuity of the norms, allow us to pass to the limit in \eqref{ene_uni_epsilon}, yielding
\begin{equation*}
\begin{aligned}
& \sup_{t \in [0,\mathcal{T}]}\left( \|(u,v,T, q)\|_{H^1}^2 + \|T\|_\infty \right)
   + \int_0^\mathcal{T} \|(\Delta u, \Delta v, \Delta T, \partial_t u, \partial_t v, \partial_t T, \partial_t q)\|_2^2 dt \\
   & + \int_0^\mathcal{T} \|\nabla u\|_{\infty}+  \|\nabla v\|_{\infty} dt \leq  \liminf_{\epsilon \rightarrow 0} \Bigg(\sup_{t \in [0, \mathcal{T}]}\left(\|(u_\epsilon,v_\epsilon,T_{\epsilon}, q_\epsilon)\|_{H^1}^2 + \|T_\epsilon\|_\infty \right) 
   \\
   & + \int_0^\mathcal{T} \|(\Delta u_\epsilon, \Delta v_\epsilon, \Delta T_\epsilon, \partial_t u_\epsilon, \partial_t v_\epsilon, \partial_t T_{\epsilon}, \partial_t q_\epsilon)\|_2^2 dt + \int_0^\mathcal{T} \big( \|\nabla u_\epsilon\|_{\infty} + \|\nabla v_\epsilon\|_{\infty} \big) dt \Bigg) \leq C.
\end{aligned}
\end{equation*}

It remains to establish the pointwise convergence as $\epsilon \to 0$ of the temperature and moisture equations \eqref{eq_T} and \eqref{eq_q}.

\subsection*{Convergence of the temperature and the moisture equations}
In this section, we prove the validity of the equations \eqref{eq_T} and \eqref{eq_q} almost everywhere on $\mathbb{R}^2 \times [0, \mathcal{T}]$. 
In other words, we need to show that the limit solution $(u,v,T,q)$ satisfies

\begin{itemize}
    \item[\textbf{I case:}] On $\left\{ (x,t) \in \mathbb{R}^2 \times [0, \mathcal{T}]: q-q_s < 0 \right\}$ it holds a.e that
    \begin{align*}
        & \partial_t q + u \nabla q + \bar{Q}\dive v = 0 \\
        &  \partial_t T + u \nabla T - \Delta T -\frac{H}{\pi}\frac{\theta_0 N^2}{g} \dive v = 0.
    \end{align*}
    \item[\textbf{II case:}]On $\left\{(x,t) \in \mathbb{R}^2 \times [0, \mathcal{T}]: q-q_s > 0 \right\}$ it holds a.e that
    \begin{equation*}
        \begin{aligned}
            & \partial_t q + u \nabla q + \bar{Q}\dive v + \frac{H g}{\pi R}G^+(T)(\dive v)^-= 0 \\
            &  \partial_t T + u \nabla T - \Delta T -\frac{H}{\pi}\frac{\theta_0 N^2}{g} \dive v =   \frac{H g}{\pi R}(\dive v)^- G^+(T). \\
        \end{aligned}
    \end{equation*} 
    \item[\textbf{III case:}] On $\left\{(x,t) \in \mathbb{R}^2 \times [0, \mathcal{T}]: q-q_s = 0 \right\}$ it holds a.e that
    \begin{equation*}
        \begin{aligned}
            & \partial_t q + u \nabla q + \bar{Q}\dive v +\frac{H g}{\pi R}G^+(T)(\dive v)^-h_q= 0 \\
            &  \partial_t T + u \nabla T - \Delta T -\frac{H}{\pi}\frac{\theta_0 N^2}{g} \dive v =   \frac{H g}{\pi R}(\dive v)^- G^+(T) h_q,\\
        \end{aligned}
    \end{equation*}
\end{itemize}
where $h_q$ is the weak* limit of $\mathcal{H}_\epsilon$ and $h_q = \alpha \in [0,1]$ when $q(x,t)=q_s$. For simplicity, we will focus on the moisture equation, as the argument for the temperature equation is analogous.

\subsection*{I case}
In the same spirit of \cite{LT}, for every triplet of positive integers $j,k,l >0$ we can define:
\begin{equation*}
    \begin{aligned}
        &O^- = \left\{ (x,t) \in \mathbb{R}^2 \times (0,\infty) : q-q_s < 0\right \}, \\
        & O_{jkl}^- = \left\{ (x,t) \in B_k(0) \times (0,l): q-q_s <- \frac{1}{j} \right \}, \\
        & O^- = \bigcup_{j= 1}^{\infty} \bigcup_{k= 1}^{\infty} \bigcup_{l= 1}^{\infty} O_{jkl}^-,
    \end{aligned}
\end{equation*}
where $B_k$ is the usual disk in $\mathbb{R}^2$ of radius $k$. Let us fix $(j,k,l)$, from \eqref{a_l_eps} we deduce that $q_\epsilon \rightarrow q$ in $L^2(O_{jkl}^-)$ and thus there exists a subsequence $q_{\epsilon_n}$ such that $q_{\epsilon_n} \rightarrow q$ a.e. in $O_{jkl}^-$. Thanks to the Egoroff theorem $\forall \eta > 0 \; \exists \; E_{\eta} \subset O_{jkl}^-$  such that $|E_{\eta}| \leq \eta$ and $q_{\epsilon_n} \rightarrow q$ uniformly on $O_{jkl}^- \setminus E_{\eta}$. Therefore,
\begin{equation*}
    \forall j>0 \quad \exists \quad \bar{n} = n(j) \; \text{ such that }\; \forall n > \bar{n} \quad q_{\epsilon_n} - q_s \leq q - q_s + \frac{1}{2j} \leq -\frac{1}{2j} <0,
\end{equation*}
this implies that for $q_{\epsilon_n}$ $H_{\epsilon_n}(q_{\epsilon_n} -q_s)=0$ on $O_{jkl}^- \setminus E_{\eta}$ for $\epsilon_n$ sufficiently small ($\epsilon_n < \epsilon_{\bar{n}}$) and
\begin{equation*}
    \mathcal{G_\epsilon}:= \partial_t q_{\epsilon_n} + u_{\epsilon_n} \nabla q_{\epsilon_n} + \bar{Q} \dive v_{\epsilon_n} = 0 \quad \text{ on } \quad O_{jkl}^- \setminus E_{\eta}.
\end{equation*}
From the strong convergence of $(q_{\epsilon_n}, u_{\epsilon_n}, v_{\epsilon_n})$ in $\eqref{a_l_eps}$ we deduce 
\begin{equation*}
\begin{aligned}
    &\mathcal{G_\epsilon} \overset{}{\rightharpoonup} \mathcal{G} \quad \text{in} \quad L^2( O_{jkl}^- \setminus E_{\eta}), \\
    & \mathcal{G}:= \partial_t q + u \nabla q + \bar{Q} \dive v =0.
\end{aligned}
\end{equation*}
Since $\mathcal{G_\epsilon} = 0 \quad  a.e \quad  O_{jkl}^- \setminus E_{\eta}$ this means that $\mathcal{G} = 0 \quad  a.e \quad on \qquad  O_{jkl}^- \setminus E_{\eta}$. Thanks to Lemma \ref{egoroff} we conclude that $\mathcal{G}=0$ on $O_{jkl}^-$ for any positive $(j,k,l)$ and therefore it holds on $O^{-}$.

\subsection*{II case}
In the same way we construct $O^+$ and $O_{jkl}^+$ such that:
\begin{equation*}
    \begin{aligned}
        & O^+ = \left\{ (x,t) \in \mathbb{R}^2 \times (0, \infty): q-q_s > 0\right \}, \;  O_{jkl}^+ = \left\{ (x,t) \in B_k(0) \times (0,l) : \frac{1}{j} < q-q_s\right \}, \\
        &O^+ = \bigcup_{j= 1}^{\infty} \bigcup_{k= 1}^{\infty} \bigcup_{l= 1}^{\infty} O_{jkl}^+.
    \end{aligned}
\end{equation*}
As before, for every fixed $(j,k,l)$, the strong convergence of $q_\epsilon$ in \eqref{a_l_eps} and the Egoroff theorem imply that $\forall \eta > 0 \; \exists \; E_{\eta} \subset O_{jkl}^-$  such that $|E_{\eta}| \leq \eta$ and a subsequence $q_{\epsilon_n} \rightarrow q$ uniformly on $O_{jkl}^+ \setminus E_{\eta}$. In particular, for any fixed $\eta$, it holds that 
\begin{equation*}
    \forall j>0 \quad \exists \quad \bar{n} = n(j) \; \text{ such that }\; \forall n > \bar{n} \quad q -\frac{1}{2j} < q_{\epsilon_n} < q +\frac{1}{2j}\mbox{ and }\epsilon_n<\frac{1}{2j}.
\end{equation*}
Therefore for any $n>\bar{n}$ it holds
\begin{equation*}
    0 < \epsilon_n < \frac{1}{2j}=\frac{1}{j}- \frac{1}{2j} < q - q_s - \frac{1}{2j} \leq q_{\epsilon_n} - q_s,
\end{equation*}
and thus
\begin{equation*}
    q_{\epsilon_n} - q_s > \epsilon_n \quad \text{ on } \quad O_{jkl}^+ \setminus E_{\eta}.
\end{equation*}
This implies that for $q_{\epsilon_n}$ it holds that $H_{{\epsilon}_n}(q_{\epsilon_n} -q_s)=1$ and
\begin{equation*}
    \mathcal{G_\epsilon}:= \partial_t q_{\epsilon_n} + u_{\epsilon_n} \nabla q_{\epsilon_n} + \bar{Q} \dive v_{\epsilon_n} + \frac{H}{\pi} G^+(T_{\epsilon_n})( \dive v_{\epsilon_n})^+= 0 \quad \text{ on } \quad O_{jkl}^+ \setminus E_{\eta}.
\end{equation*}
From the strong convergence of $(q_{\epsilon_n}, u_{\epsilon_n}, v_{\epsilon_n})$ in \eqref{a_l_eps} we deduce that
\begin{equation*}
\begin{aligned}
    &\mathcal{G_\epsilon} \overset{}{\rightharpoonup} \mathcal{G} \quad \text{in} \quad L^2( O_{jkl}^- \setminus E_{\eta}), \\
    & \mathcal{G}:= \partial_t q + u \nabla q + \bar{Q} \dive v + \frac{H g}{\pi R} G^+(T)( \dive v)^+=0.
\end{aligned}
\end{equation*}
Since $\mathcal{G_\epsilon} = 0 \quad  a.e \quad  O_{jkl}^+ \setminus E_{\eta}$ this means that $\mathcal{G} = 0 \quad  a.e \quad  O_{jkl}^+ \setminus E_{\eta}$. Thanks to Lemma \ref{egoroff} we conclude that $\mathcal{G}=0$ on $O_{jkl}^+$ for any positive $(j,k,l)$ and therefore it holds on $O^{+}$.

\subsection*{III case}
As in the previous case we define
\begin{equation*}
    \begin{aligned}
        & O = \left\{ (x,t) \in \mathbb{R}^2 \times (0,\infty) : q-q_s = 0\right \}, \quad O_{kl} = \left\{ (x,t) \in B_k(0) \times (0,l) : q=q_s, \right \}, \\
        & O = \bigcup_{k= 1}^{\infty} \bigcup_{l= 1}^{\infty} O_{kl}.
    \end{aligned}
\end{equation*}

Let us also define
\begin{equation*}
    \mathcal{G_\epsilon}:= \partial_t q_{\epsilon_n} + u_{\epsilon_n} \nabla q_{\epsilon_n} + \bar{Q} \dive v_{\epsilon_n} + \frac{H}{\pi} H_{\epsilon_n}(q_{\epsilon_n} -q_s) G^+(T_{\epsilon_n})( \dive v_{\epsilon_n})^+= 0 \quad \text{ on } \quad O_{kl}.
\end{equation*}
By using the convergences in \eqref{conv_eps} it is easy to obtain that 

\begin{equation*}
\begin{aligned}
    &\mathcal{G_\epsilon} \overset{}{\rightharpoonup} \mathcal{G} \quad \text{in} \quad L^2( O_{kl}), \\
    & \mathcal{G}:= \partial_t q + u \nabla q + \bar{Q} \dive v + \frac{H}{\pi} h_q G^+(T)( \dive v)^+=0.
\end{aligned}
\end{equation*}
Since $\mathcal{G_\epsilon} = 0 \quad  a.e \quad  O_{kl}$, by using again Lemma \ref{egoroff} we deduce that  this means that $\mathcal{G} = 0 \quad  a.e \quad  O$. Finally, since $h_q$ is the weak start limit of $H_{\epsilon_n}(q_{\epsilon_{n}}-q_s)\in [0,1]$ it also holds that $h_q = \alpha \in [0,1]$.

\subsection{Proof of Theorem \ref{teo:uniqueness} (Uniqueness of global strong solutions)}\label{uniqueness}
In this section we establish the uniqueness of strong global solutions to the system \eqref{tro_model_fin}. Let $(u_1,v_1,T_{1},q_1)$ and $(u_2,v_2,T_{2},q_2)$ be two strong solutions of \eqref{tro_model_fin} in the sense of Definition \ref{def_strong}, corresponding to the same initial data $(u_0,v_0,T_{0},q_0)$ on the time interval $[0, \mathcal{T})$ for any $0<\mathcal{T}<\infty$. We define the difference as
\begin{equation*}
    (\delta u, \delta v, \delta T, \delta q)= (u_1,v_1,T_{1},q_1)- (u_2,v_2,T_{2},q_2),
\end{equation*}
thus we have
\begin{equation}\label{eq_diff_sol}
\left\{
    \begin{aligned}
     &\partial_t \delta u + u_1 \nabla \delta u + \delta u \nabla u_2 - \Delta \delta u + \nabla \delta p + \dive (v_1 \otimes \delta v) + \dive (\delta v \otimes v_2) = 0 \\
    &\partial_t \delta v + u_1 \nabla \delta v + \delta u \nabla v_2 - \Delta \delta v + v_1 \nabla \delta u + \delta v \nabla u_2 = \frac{H}{\pi} \frac{g}{\theta_0} \nabla \delta T \\
    &\dive \delta u  = 0 \\
    &\partial_t \delta T + u_1 \nabla \delta T + \delta u \nabla T_{2} - \Delta \delta T - \frac{H}{\pi}\frac{N^2 \theta_0}{g} \dive  \delta v  = \frac{H g}{\pi R}(\dive v_1)^- h_{q_1} G(T_1)^+  \\
    & \hspace{8.5cm}- \frac{H g}{\pi R} (\dive v_2)^- h_{q_2} G(T_2)^+ \\
    &\partial_t \delta q + u_1 \nabla \delta q + \delta u \nabla q_2 + \bar{Q} \dive \delta v  = -  \frac{H g}{\pi R}(\dive v_1)^- h_{q_1}G(T_1)^+ \\
    & \hspace{6cm}+ \frac{H g}{\pi R} (\dive v_2)^- h_{q_2} G(T_2)^+. 
\end{aligned}\right.
\end{equation}
We recall  the assumption that
\begin{equation*}
    (q_1-q_s)(q_2-q_s)\ge 0
\quad \text{a.e. in } \mathbb{R}^2\times(0,\mathcal{T}),
\end{equation*}
that is, the quantities $q_1-q_s$ and $q_2-q_s$ possibly have different sign on sets of measure zero.

As usual, we multiply the first two equations in \eqref{eq_diff_sol} by $(\delta u, \delta v)$ and integrate over $\mathbb{R}^2$. After standard manipulations we obtain 
\begin{equation}\label{pt_1}
    \begin{aligned}
        & \frac{1}{2} \frac{d}{dt} \left( \|\delta u\|_2^2 + \|\delta v\|_2^2  \right) + \frac{1}{2}\|\nabla \delta u\|_2^2 + \frac{3}{4}\|\nabla \delta v\|_2^2 \\
        & \leq C \int_{\mathbb{R}^2}[(|\nabla u_2|+ |\nabla v_2| + |\nabla v_1| + |v_1|^2 + |v_2|^2] (|\delta u|^2 + |\delta v|^2)  dx + C \|\delta T\|_2^2.
    \end{aligned}
\end{equation}
Regarding the temperature and moisture equations, we decompose the domain $\Omega = \mathbb{R}^2 \times (0,\infty)$ as follows:
\begin{equation*}
    \Omega = \sum_{i=1}^7\Omega_i,
\end{equation*}
where
\begin{equation*}
\begin{aligned}
   &\Omega_1  = \left\{ q_1 - q_s < 0 \right\} \cap \left\{ q_2 - q_s < 0 \right\},\qquad \Omega_2=   \left\{ q_1 - q_s < 0 \right\} \cap \left\{ q_2 - q_s = 0 \right\}, \\
   &\Omega_3=   \left\{ q_1 - q_s = 0 \right\} \cap \left\{ q_2 - q_s < 0 \right\}, \qquad \Omega_4 =   \left\{ q_1 - q_s = 0 \right\} \cap \left\{ q_2 - q_s = 0 \right\}, \\
   & \Omega_5=   \left\{ q_1 - q_s = 0 \right\} \cap \left\{ q_2 - q_s > 0 \right\}, \qquad \Omega_6=   \left\{ q_1 - q_s > 0 \right\} \cap \left\{ q_2 - q_s = 0 \right\},\\
   &\Omega_7=   \left\{ q_1 - q_s > 0 \right\} \cap \left\{ q_2 - q_s > 0 \right\}.
\end{aligned}
\end{equation*}
We notice that on the set $\Omega_1$ since $h_{q_1} = h_{q_2}=0$ the equations for $\delta T, \delta q$ satisfy
\begin{equation}\label{delta_ome_1}
    \begin{aligned}
        & \partial_t \delta T + u_1 \nabla \delta T + \delta u \nabla T_{2} - \Delta \delta T - \frac{H}{\pi}\frac{N^2 \theta_0}{g} \dive  \delta v=0, \\
        & \partial_t \delta q + u_1 \nabla \delta q + \delta u \nabla q_2 + \bar{Q} \dive \delta v  =0.
    \end{aligned}
\end{equation}
On the set $\Omega_2$ since $q_2=q_s$ then $\partial_t q_2 + u \nabla q_2=0$ for a.e. $(x,t)$. Thus, from the equation of $q_2$ we have
\begin{equation*}
    \bar{Q} \dive v_2 = - \frac{Hg}{\pi R}(\dive v_2)^- G^+(T_2)\alpha,
\end{equation*}
where $h_{q_2}=\alpha \in (0,1)$. This, in particular, implies that on $\Omega_2$,  $\dive v_2 = 0$. This type of argument will be applied multiple times in the proof, even if it is not explicitly mentioned each time. Finally, on $\Omega_2$ the equations for $\delta T, \delta q$ satisfy
\begin{equation}\label{delta_ome_2}
    \begin{aligned}
        & \partial_t \delta T + u_1 \nabla \delta T + \delta u \nabla T_{2} - \Delta \delta T - \frac{H}{\pi}\frac{N^2 \theta_0}{g} \dive  \delta v=0, \\
        & \partial_t \delta q + u_1 \nabla \delta q + \delta u \nabla q_2 + \bar{Q} \dive v_1 =0.
    \end{aligned}
\end{equation}
On $\Omega_3$ we have $\dive v_1=0$ and
\begin{equation}\label{delta_ome_3}
    \begin{aligned}
        & \partial_t \delta T + u_1 \nabla \delta T + \delta u \nabla T_{2} - \Delta \delta T + \frac{H}{\pi}\frac{N^2 \theta_0}{g} \dive v_2= 0, \\
        & \partial_t \delta q + u_1 \nabla \delta q + \delta u \nabla q_2 - \bar{Q} \dive v_2= 0.
    \end{aligned}
\end{equation}
On $\Omega_4$ we have $\dive v_1=0=\dive v_2$ and
\begin{equation}\label{delta_ome_4}
    \begin{aligned}
        & \partial_t \delta T + u_1 \nabla \delta T + \delta u \nabla T_{2} - \Delta \delta T = 0, \\
        & \partial_t \delta q + u_1 \nabla \delta q + \delta u \nabla q_2 = 0.
    \end{aligned}
\end{equation}
On $\Omega_5$ we have $\dive v_1=0$ and
\begin{equation}\label{delta_ome_5}
    \begin{aligned}
        & \partial_t \delta T + u_1 \nabla \delta T + \delta u \nabla T_{2} - \Delta \delta T + \frac{H}{\pi}\frac{N^2 \theta_0}{g} \dive v_2= -\frac{Hg}{\pi R}(\dive v_2)^-G^+(T_2), \\
        & \partial_t \delta q + u_1 \nabla \delta q + \delta u \nabla q_2 - \bar{Q} \dive v_2= \frac{Hg}{\pi R}(\dive v_2)^-G^+(T_2).
    \end{aligned}
\end{equation}
On $\Omega_6$ we have $\dive v_2=0$ and

\begin{equation}\label{delta_ome_6}
    \begin{aligned}
        & \partial_t \delta T + u_1 \nabla \delta T + \delta u \nabla T_{2} - \Delta \delta T - \frac{H}{\pi}\frac{N^2 \theta_0}{g} \dive v_1= \frac{Hg}{\pi R}(\dive v_1)^-G^+(T_1), \\
        & \partial_t \delta q + u_1 \nabla \delta q + \delta u \nabla q_2 + \bar{Q} \dive v_1= -\frac{Hg}{\pi R}(\dive v_1)^-G^+(T_1).
    \end{aligned}
\end{equation}
On $\Omega_7$ we have,

\begin{equation}\label{delta_ome_7}
    \begin{aligned}
         & \partial_t \delta T + u_1 \nabla \delta T + \delta u \nabla T_{2} - \Delta \delta T - \frac{H}{\pi}\frac{N^2 \theta_0}{g} \dive \delta v \\
         & \hspace{+ 3.cm}= \frac{Hg}{\pi R}(\dive v_1)^-G^+(T_1) - \frac{Hg}{\pi R}(\dive v_2)^-G^+(T_2),\\
        & \partial_t \delta q + u_1 \nabla \delta q + \delta u \nabla q_2 + \bar{Q} \dive \delta v = -\frac{Hg}{\pi R}(\dive v_1)^-G^+(T_1) + \frac{Hg}{\pi R}(\dive v_2)^-G^+(T_2).
    \end{aligned}
\end{equation}
We fix the following notation,
\begin{equation*}
    \delta H := H_1 - H_2 =  (\dive v_1)^-G^+(T_1) - (\dive v_2)^-G^+(T_2).
\end{equation*}
Therefore, combining \eqref{delta_ome_1} - \eqref{delta_ome_7}, we obtain
\begin{equation}\label{equation_delta_T}
    \begin{aligned}
        \partial_t \delta T + u_1 \nabla \delta T + \delta u \nabla T_{2}  = &\frac{H}{\pi}\frac{N^2 \theta_0}{g}(\dive \delta v \chi_{\Omega_1}  + \dive \delta v \chi_{\Omega_7})  + \frac{H}{\pi}\frac{N^2 \theta_0}{g}(\dive v_1 \chi_{\Omega_2} + \dive v_1 \chi_{\Omega_6} ) \\
        & + \frac{H}{\pi}\frac{N^2 \theta_0}{g}(-\dive v_2 \chi_{\Omega_3} - \dive v_2 \chi_{\Omega_5} ) - \frac{Hg}{\pi R}(\dive v_2)^-G^+(T_2) \chi_{\Omega_5} \\
        & + \frac{Hg}{\pi R}(\dive v_1)^-G^+(T_1)  \chi_{\Omega_6} + \frac{Hg}{\pi R}\delta H\chi_{\Omega_7}\\
        & = \frac{H}{\pi}\frac{N^2 \theta_0}{g} (\dive \delta v - \dive \delta v \chi_{\Omega_4}) \\
        & + \frac{H}{\pi}\frac{N^2 \theta_0}{g}(\dive v_2  \chi_{\Omega_2} - \dive v_1 \chi_{\Omega_3} - \dive v_1 \chi_{\Omega_5} + \dive v_2 \chi_{\Omega_6}) \\
        & + \frac{Hg}{\pi R}(\delta H -  \delta H\chi_{\Omega_1} - \delta H\chi_{\Omega_2} - \delta H\chi_{\Omega_3}- \delta H\chi_{\Omega_4})\\
        & + \frac{Hg}{\pi R}(- H_1\chi_{\Omega_5}  + H_2\chi_{\Omega_6}) \\
        & = I_1 + I_2 + I_3 + I_4.
    \end{aligned}
\end{equation}

\begin{equation}\label{equation_delta_q}
    \begin{aligned}
        \partial_t \delta q + u_1 \nabla \delta q + \delta u \nabla q_2  = & - \bar{Q}(\dive \delta v \chi_{\Omega_1}   + \dive \delta v \chi_{\Omega_7}+ \dive v_1 \chi_{\Omega_2} + \dive v_1 \chi_{\Omega_6})\\ 
        & + \bar{Q}(\dive v_2 \chi_{\Omega_3} + \dive v_2 \chi_{\Omega_5} )+ \frac{Hg}{\pi R}(\dive v_2)^-G^+(T_2) \chi_{\Omega_5} \\
        &- \frac{Hg}{\pi R}(\dive v_1)^-G^+(T_1)  \chi_{\Omega_6} - \frac{Hg}{\pi R} \delta H\chi_{\Omega_7}\\
        & = - \bar Q (\dive \delta v - \dive \delta v \chi_{\Omega_4}) \\
        & - \bar Q (  \dive v_2  \chi_{\Omega_2} - \dive v_1 \chi_{\Omega_3} + \dive v_2 \chi_{\Omega_6}  - \dive v_1 \chi_{\Omega_5})  \\
        & - \frac{Hg}{\pi R}(\delta H -  \delta H\chi_{\Omega_1} - \delta H\chi_{\Omega_2} - \delta H\chi_{\Omega_3}- \delta H\chi_{\Omega_4})\\
        & - \frac{Hg}{\pi R}( - H_1\chi_{\Omega_5} + H_2\chi_{\Omega_6}) \\
        &= J_1 + J_2 + J_3 + J_4.
    \end{aligned}
\end{equation}
Let us start by multiplying equation \eqref{equation_delta_T} by $\delta T$ and integrate over $\mathbb{R}^2$,
\begin{equation}\label{eq_T_final}
    \begin{aligned}
        \frac{1}{2}\frac{d}{dt} \|\delta T\|_2^2+ \|\nabla \delta T\|_2^2 & = - \int_{\mathbb{R}^2} \delta u \nabla T_2 \delta T dx + \frac{H}{\pi}\frac{N^2 \theta_0}{g} \int_{\mathbb{R}^2} (\dive \delta v - \dive \delta v \chi_{\Omega_4})\delta Tdx \\
        & + \frac{Hg}{\pi R} \int_{\mathbb{R}^2}(\delta H -  \delta H\chi_{\Omega_1} - \delta H\chi_{\Omega_2} - \delta H\chi_{\Omega_3}- \delta H\chi_{\Omega_4})\delta T dx\\
        & \leq \|\delta u\|_\infty \|\nabla T_2\|_2 \|\delta T\|_2  + C(1 + \|\nabla v_2\|_4^4)\|\delta T\|_2^2 \\
        & + \frac{1}{4} \|\nabla \delta v\|_2^2 + \frac{1}{4}\|\nabla \delta T \|_2^2, 
    \end{aligned}
\end{equation}
since $I_2 = 0$, $H_2\chi_{\Omega_6}=0$ and $H_1\chi_{\Omega_5} = 0$. 

Finally, we multiply \eqref{equation_delta_q} by $\delta q$ and integrate over $\mathbb{R}^2$, we have
\begin{equation}\label{eq_q_final}
    \begin{aligned}
        \frac{1}{2}\frac{d}{dt} \|\delta q\|_2^2 & \leq \|\delta u\|_\infty \|\nabla q_2\|_2 \|\delta q\|_2 + C(1+ \|\nabla v_2\|_4^4)(\|\delta T\|_2^2 + \|\delta q\|_2^2) \\
        & \frac{1}{4}\|\nabla \delta v\|_2^2 + \frac{1}{4}\|\nabla \delta T\|_2^2
    \end{aligned}
\end{equation}
since $J_2 = 0$, $H_2\chi_{\Omega_6}=0$ and $H_1\chi_{\Omega_5} = 0$. At this stage, by combining \eqref{pt_1}, \eqref{eq_T_final}, and \eqref{eq_q_final}, and proceeding as in Subsection \ref{unique_1}, we can apply Lemma \ref{uniqueness} to conclude the proof.

\begin{acknowledgments}
The authors gratefully acknowledge the partial support by the Gruppo Na\-zio\-na\-le per l’Analisi Matematica, la Probabilit\`a e le loro
Applicazioni (GNAMPA) of the Istituto Nazionale di Alta Matematica (INdAM),  by the PRIN2022
``Classical equations of compressible fluids mechanics: existence and
properties of non-classical solutions'' and by the PRIN2022-PNRR ``Some
mathematical approaches to climate change and its impacts.''
\end{acknowledgments}

\vspace{1cm}
\textbf{Data Availability.} The authors declare that data sharing is not applicable to this article as no datasets were generated or analysed.
\\
\\
\textbf{Conflict of interest.} The authors declare that they have no conflict of interest.

\end{document}